\documentclass{amsart}
\usepackage[utf8]{inputenc}
\pdfoutput=1

\usepackage[margin=1.25in]{geometry}

\usepackage{amsmath,color,graphicx,amssymb}
\usepackage{amsthm}
\usepackage{amsmath}
\usepackage{amssymb, amsfonts, amsxtra, mathrsfs, mathtools} 
\usepackage{tikz}
\usetikzlibrary{calc,matrix,arrows,shapes.geometric,positioning,decorations.pathmorphing}
\usepackage{tikz-cd}
\usepackage{tikzsymbols}

\usepackage{hyperref}

\usepackage{float}
\usepackage{indentfirst}
\usepackage{fancyhdr}

\usepackage{yhmath}

\newtheorem{thm}{Theorem}[section] 

\newtheorem{lemma}[thm]{Lemma}     

\newtheorem{cor}[thm]{Corollary}

\newtheorem{prop}[thm]{Proposition}

\theoremstyle{definition}
\newtheorem{defn}[thm]{Definition}

\theoremstyle{definition}
\newtheorem{ex}[thm]{Example}

\theoremstyle{definition}
\newtheorem{remk}[thm]{Remark}

\theoremstyle{definition}
\newtheorem{constr}[thm]{Construction}

\theoremstyle{definition}

\theoremstyle{definition}
\newtheorem{notat}[thm]{Notation}

\theoremstyle{definition}
\newtheorem{terminology}[thm]{Terminology}

\numberwithin{equation}{thm}

\allowdisplaybreaks

\title{\large{\uppercase{Powers of Edge Ideals with Linear Quotients}}}

\author[E.~Basser]{Etan Basser}
\address{Department of Mathematics, Yale University, New Haven, CT}
\email{etan.basser@yale.edu}

\author[R.~Diethorn]{Rachel Diethorn}
\address{Department of Mathematics,
 Oberlin College, Oberlin, OH}
\email{rdiethor@oberlin.edu}

\author[R.~Miranda]{Robert Miranda}
\address{Department of Mathematics, UCLA, Los Angeles, CA}
\email{robertmiranda@math.ucla.edu}

\author[M.~Stinson-Maas]{Mario Stinson-Maas}
\address{Department of Mathematics,
 Oberlin College, Oberlin, OH}
\email{mstinson@oberlin.edu}

\date{}

\begin{document}

\thanks{The second and fourth authors were partially supported by National Science Foundation Award No. 2418637. 
} 

\begin{abstract}
We prove that second and higher powers of the edge ideals of anticycles admit linear quotient orderings, although the edge ideals themselves do not, thus resolving an open question of Hoefel and Whieldon in the affirmative and providing the first class of gap-free graphs whose edge ideals satisfy this property on their powers.  We also construct an explicit and straightforward linear quotient ordering for any power of a quadratic monomial ideal which admits linear quotients.  This expands on a well-known result of Herzog, Hibi, and Zheng.  As a consequence, we give explicit formulas for the projective dimension and Betti numbers of the edge ideals of whisker graphs.
\end{abstract}

\maketitle

\section{Introduction}

The problem of describing resolutions of powers of \textit{edge ideals} (that is, square-free quadratic monomial ideals) has been a topic of intense interest in recent years; see for example \cite{banerjee15}, \cite{BBKH}, \cite{bht15}, \cite{bigdeli}, \cite{herzogetalindex}, \cite{CEFMMSS}, \cite{n18}, \cite{erey}, \cite{FMO12}, \cite{HHZ04}, \cite{HW11}, \cite{mv}, \cite{mfy17}, \cite{nevo}, \cite{Nevo2009LINEARRO}, \cite{NP13}, and \cite{s18}.  A central theme in this vast body of work is understanding when
such resolutions are linear; that is, when the entries in the matrices representing the differentials are linear forms.
Motivated by work of Herzog, Hibi, and Zheng in \cite{HHZ04}, we approach this problem via the linear quotient property, which is defined as follows.

\begin{defn}
An ideal $I$ of a standard graded polynomial ring $P=\mathsf{k}[x_1,\dots,x_n]$ has \textit{linear quotients} if for some ordering of a minimal set of generators $m_1,\dots,m_r$ of $I$, each ideal quotient
\begin{displaymath}
((m_1,...,m_{i-1}):(m_{i}))
\end{displaymath}
for $i=2,\dots , r$, is generated by some subset of the variables $x_1,\dots,x_n$.
\end{defn}

The linear quotient property, first introduced by Herzog and Takayama in \cite{ht02}, is known to impose strong homological restrictions on the ideal $I$.  In fact, for monomial ideals that admit linear quotients, one can even build a linear resolution explicitly using the well-known iterated mapping cone construction; see \cite{ht02}.  Even more can be said in the case of quadratic monomial ideals (e.g. edge ideals) as we see in the following result of Herzog, Hibi, and Zheng.  

\begin{thm} (\cite[Theorem 3.2]{HHZ04})\label{HHZ}
The following conditions are equivalent for a quadratic monomial ideal $I$ of the standard graded polynomial ring $P=\mathsf{k}[x_1,\dots,x_n]$:
\begin{itemize}
\item[(a)]  $I$ has a linear resolution;
\item[(b)]  $I$ has linear quotients;
\item[(c)]  $I^n$ has a linear resolution for all $n\geq 1$.
\end{itemize}
\end{thm} 

Given this result, a natural question is whether the powers of an ideal having linear quotients have linear quotients as well.  The answer to this question is known to be negative for non-quadratic monomial ideals.  For example, in \cite[Example 4.3]{CH03}, Conca and Herzog provide an example of an ideal $I$ generated by degree 3 monomials that has linear quotients, but such that $I^2$ does not.  For quadratic monomial ideals, however, the question was answered in the affirmative in \cite[Theorems 10.1.9 and 10.2.5]{HHbook} building from work in \cite{HHZ04}; see also \cite[Theorem 2.6]{dali15}.  The theorems above show the existence of such linear quotient orderings using the so-called x-condition on the Gr{\"o}bner basis of the defining ideal of the Rees algebra of $I$. 


In Theorem \ref{mainthm}, we construct an  explicit linear quotient ordering for any power of  a quadratic monomial ideal $I$ which admits linear quotients, depending only on the linear quotient ordering of the ideal $I$ itself. This gives a constructive proof, motivated by the structure of the corresponding graph, of the above mentioned result from \cite{HHbook}.
A different ordering is given in \cite{ficarra24}; see also \cite{banerjee15}.
We illustrate the utility of our linear quotient ordering by providing explicit formulas for the projective dimension and Betti numbers of the powers of the edge ideal of what we call a \textit{whisker graph}, generalizing work of  Ferr\'o, Murgia, and Olteanu in \cite{FMO12} and complementing results such as \cite[Corollary 3.6]{diethorn} and \cite[Corollary 2.6]{rvt}. 
 Furthermore, our linear quotient ordering inspires Example \ref{counterex}, which gives a counterexample to \cite[Conjecture 4.1]{EFHHKM}.

Conversely, one may ask if $I^k$ has linear quotients for some $k > 1$, must $I$ have linear quotients? The answer to this question is negative, as Hoefel and Whieldon proved in \cite[Theorem 4.2]{HW11} that for an anticycle graph on at least five vertices, the square of the edge ideal has linear quotients even though the edge ideal itself does not.  However, their ordering did not extend 
to higher powers, leaving the open question of whether higher order powers of the edge ideal of the anticycle graph admit linear quotients; see \cite[Question 5.1]{HW11}. 



We answer this question in the affirmative in Theorem \ref{anticycle-theorem}; that is, we prove that the second and higher order powers of the edge ideal of the anticycle graph on at least 5 vertices admit linear quotients by constructing an explicit and surprisingly simple linear quotient ordering.  This provides the first class of so-called gap-free graphs whose edge ideals do not have linear quotients, but whose second and higher powers have linear quotients.  Such classes of graphs are of significant interest, given their connection to Nevo and Peeva's open conjecture in \cite{NP13} that sufficiently high powers of the edge ideal of a gap-free graph have linear resolutions, and the related open conjecture in \cite{EFHHKM} that if the $N$th power of an edge ideal $I$ has linear quotients, then so does $I^k$ for all $k\geq N$.
For recent work related to these conjectures see, for example, \cite{banerjee15}, \cite{BBKH}, \cite{bigdeli}, \cite{n18}, \cite{erey}, \cite{mv}, \cite{nevo}, and \cite{Nevo2009LINEARRO}.

We now outline the contents of this paper.  In Section 2 we collect some preliminary definitions and facts  regarding the linear quotient property of monomial ideals that we will use throughout the paper.  In Section 3 we prove in Theorem \ref{mainthm} that the powers of a quadratic monomial ideal with linear quotients also admit linear quotients; our linear quotient ordering appears in Construction \ref{ordering}.  We also provide a counterexample to \cite[Conjecture 4.1]{EFHHKM}; see Example \ref{counterex} and Remark \ref{counterrem}.  In Section 4 we use this linear quotient ordering to provide explicit formulas for the projective dimension and Betti numbers of powers of edge ideals of whisker graphs.  In Section 5 we prove in Theorem \ref{anticycle-theorem} that the higher order powers of the edge ideal of the anticycle graph have linear quotients; our ordering appears in Construction \ref{anticycle-order}.  Finally, in Section 6, we examine the problem of finding linear quotient orderings from a computational perspective using new and original methods on Macaulay2 \cite{M2} and highlight some relevant examples that support our work. 
 Our code used to execute these computations as well as relevant documentation can be found in \cite{MSM24}.

\section{Preliminaries}

In this section we collect some preliminary definitions and facts we use throughout the paper, including the definition of an edge ideal, a useful fact for working with linear quotient orderings, and formulas for the projective dimension and Betti numbers of ideals that admit linear quotient orderings.  Establishing some notation to be used throughout the section, let $P=\mathsf{k}[x_1,\dots,x_n]$ be a standard graded polynomial ring over a field $\mathsf{k}$.  We begin by recalling the notion of an edge ideal. 

\begin{defn}
Let $G=(V,E)$ be a simple graph (that is, with no loops nor multiple edges) on vertices $V=\{x_1,\dots,x_n\}$.  The \textit{edge ideal} associated to $G$ is the $P$-ideal
\begin{align*}
I(G)=(x_ix_j\,|\,\{x_i,x_j\}\in E).
\end{align*}
\end{defn} 

\begin{remk}\label{quadmon}
    It is standard in the literature on edge ideals to assume that the underlying graph has no loops; that is, no edges of the form $\{x_i,x_i\}$.  Thus edge ideals are precisely square-free quadratic monomial ideals.  However, each quadratic monomial ideal naturally corresponds to a graph $G$ that may contain loops (but no multiple edges).  Abusing notation slightly, we write such ideals as $I(G)$.  
\end{remk}

Next we state a standard result on the linear quotient property which we will use throughout this paper; see for example \cite[Lemma 8.2.3]{HHbook}. 

\begin{lemma}\label{HWfact}
    Let $I=(m_1,\dots,m_r)$ be a monomial ideal of $P$.  The ordering $m_1,\dots,m_r$ yields linear quotients for $I$ if and only if for all $i,j\in\{1,\dots,r\}$ with $j<i$, there exists some $h<i$ (possibly equal to $j$), such that for some $x\in\{x_1,\dots,x_n\}$, the following hold:
    \begin{align*}
        \frac{m_h}{\mathrm{gcd}\,(m_h, m_i)}=x\quad\text{and}\quad x\,\Big|\,\frac{m_j}{\mathrm{gcd}\,(m_j, m_i)}.
    \end{align*} 
    
\end{lemma}



For an ideal that has linear quotients, one can obtain explicit formulas for its projective dimension and Betti numbers directly from the iterated mapping cone construction; see \cite{ht02}.  The formulas are given in the following result which appears in \cite[Corollary 2.7]{SV09}, \cite{ht02}, and \cite[Corollary 8.2.2]{HHbook}. 

\begin{prop}\label{pd&betti}
Let $I$ be a monomial ideal of $P$ with linear quotient ordering $I=(m_1,\dots,m_r)$.  Define $\nu_1=0$ and let $\nu_j$ be the minimal number of generators of the ideal $(m_1,\dots,m_{j-1}):(m_j)$, for $j=2,\dots, r$.  Then the projective dimension and Betti numbers of $I$ are given by
\begin{align*}
    \mathrm{pd}(I)=\mathrm{max}\{\nu_j\,|\,1\leq j\leq r\}
\quad\text{and}\quad
    \beta_i(I)=\sum_{j=1}^r {\nu_j\choose i}, \quad\text{for all}\quad 0\leq i\leq\mathrm{pd}(I).
\end{align*}
\end{prop}

We use Proposition \ref{pd&betti} to obtain explicit formulas for the projective dimension and the Betti numbers of the powers of the edge ideals of whisker graphs in Section 5.

\section{Powers of quadratic monomial ideals with linear quotients}

In this section we prove that the powers of a quadratic monomial ideal with linear quotients must also admit linear quotients by constructing an explicit linear quotient ordering.
Throughout this section, let $P=\mathsf{k}[x_1,\dots,x_n]$ be a standard graded polynomial ring over a field $\mathsf{k}$ and let $G$ be a graph with vertices $x_1,\dots,x_n$ and with no multiple edges.  

For the quadratic monomial ideal $I(G)=(m_1,\dots, m_r)\subseteq P$ corresponding to $G$ (see Remark \ref{quadmon}),
we are interested in relating the generators of $I(G)^k$ to the generators of $I(G)$ in order to produce a linear quotient ordering of $I(G)^k$. 
Notice that the set of formal combinations 
\[
    \mathcal{S}(G)^k=\{ m_1^{\alpha_1} \dotsc\,  m_r^{\alpha_r} \ | \ \alpha_1 + \dots + \alpha_r = k\}
\]
is a generating set for $I(G)^k$, but in many cases contains repetitions and thus is not a minimal generating set. 
This occurs in the case of cyclic graphs, as we illustrate in the next example.

\begin{ex}\label{cycle example}
Consider the cycle $\mathcal{C}_4$ as pictured below.

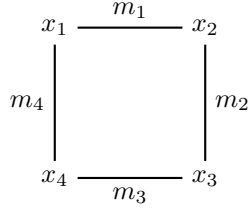
\begin{center}
\begin{tikzpicture}
    \node (x1) at (0,0) {$x_1$};
    \node (x2) at (2,0) {$x_2$};
    \node (x3) at (2,-2) {$x_3$};
    \node (x4) at (0,-2) {$x_4$};
    
    \path [thick]
    (x1) edge node [above] {$m_1$} (x2)
    (x2) edge node [right] {$m_2$} (x3)
    (x3) edge node [below] {$m_3$} (x4)
    (x4) edge node [left] {$m_4$} (x1);
\end{tikzpicture}
\end{center}

\noindent The edge ideal $I(\mathcal{C}_4)$ has generators $m_1, m_2, m_3, m_4$, but the products $m_1 m_3=x_1x_2x_3x_4$ and $m_2 m_4=x_1x_2x_3x_4$ produce the same generator of $I_{\mathcal{C}_4}^2$.
\end{ex}

We will need to handle the possibility of repetitions in $\mathcal{S}(G)^k$ carefully in our linear quotient ordering in Construction \ref{ordering}.  Before this construction, we establish some terminology and notation we will use throughout this section and the rest of the paper.

\begin{terminology}\label{decomp}
    Each element $M\in\mathcal{S}(G)^k$ can also be expressed in terms of the variables:
\begin{align*}
    M=m_1^{\alpha_1} \dotsc\,  m_r^{\alpha_r}=x_1^{a_1} \dotsc\,  x_n^{a_n}
\end{align*}
for some list of nonnegative integers $a=(a_1,\dots,a_n)$.  Throughout this paper we often refer to $x_1^{a_1} \dotsc\,  x_n^{a_n}$ as the \textit{vertex decomposition} of $M$ and to $m_1^{\alpha_1} \dotsc\,  m_r^{\alpha_r}$ as an \textit{edge decomposition} of $M$.  As we see in Example \ref{cycle example}, edge decompositions of $M$ need not be unique.  For a fixed edge decomposition $M=m_1^{\alpha_1} \dotsc\,  m_r^{\alpha_r}$, we call $m_i$ a \textit{formal edge} of $M$ if $\alpha_i>0$.  Furthermore, we say that $m_i$ is an \textit{edge} of $M$ whenever $m_i$ is a formal edge for for at least one edge decomposition of $M$.  

Whenever $m_i$ is formal edge of $M$, we use the following notation
\begin{align*}
    \frac{1}{m_i} M:=m_1^{\alpha_1}\dotsc m_i^{\alpha_i-1}\dotsc\, m_r^{\alpha_r}.
\end{align*}
Similarly, we denote by ${m_i} M$ the monomial $m_1^{\alpha_1}\dotsc\, m_i^{\alpha_i+1}\dotsc\, m_r^{\alpha_r}$.  Finally, we say that a variable $x_j$ is \textit{incident to} an an edge $m_i$ whenever $x_j$ divides $m_i$ in the polynomial ring $P$.  
\end{terminology}

With these considerations in mind, we are now ready to construct an ordering of the generators of $I(G)^k$ for a quadratic monomial ideal $I(G)$ with linear quotients.

\begin{constr}\label{ordering}
Let $I(G)\subseteq P$ be a quadratic monomial ideal with linear quotient ordering $m_1, \dotsc, m_r$. For $k\in\mathbb{N}$, consider the set of $\binom{r+k-1}{k}$ formal combinations 
\[
     \mathcal{S}(G)^k=\{m_1^{\alpha_1} \dotsc\, m_r^{\alpha_r} \ | \ \alpha_1 + \dotsc + \alpha_r = k\}.
\]
Next we form a list whose entries are the elements of $\mathcal{S}(G)^k$ ordered according to the \textit{reverse lexicographic (revlex)} ordering with respect to edge decompositions, denoted by $\mathcal{R}(G)^k$.

More explicitly, $\mathcal{R}(G)^k$ is ordered as follows:
\begin{align*}
    m_1^{\alpha_1} \dotsc m_r^{\alpha_r}\,\,\, \text{precedes}\,\,\, m_1^{\beta_1} \dotsc m_r^{\beta_r}&\iff\exists\, i\in\{1,\dots,r\}\,\,\,\text{such that}\,\,\,\alpha_j=\beta_j\,\,\,\text{for}\,\,\, j>i\,\,\,\text{and}\,\,\,\alpha_i<\beta_i\\
    &\iff\,\,\text{the last nonzero entry of the vector}\,\,\alpha-\beta\,\,\text{is negative},
\end{align*}
 which results in the following ordering:
 \begin{displaymath}
     \mathcal{R}(G)^k=\left(m_1^k, m_1^{k-1}m_2,\dots,m_2^k, m_1^{k-1}m_3, m_1^{k-2}m_2m_3,\dots, m_r^k\right).
 \end{displaymath}
 
As noted above, $\mathcal{R}(G)^k$ forms a generating set for $I(G)^k$, but is not minimal when it has repetitions.  Thus we call $M\in{\mathcal{R}(G)^k}$ a \textit{representative} 
if there is no element $M'$ preceding $M$ in ${\mathcal{R}(G)^k}$ such that $M$ and $M'$ have the same vertex decompositions (i.e., are equal in $P$). 

To complete our construction, we define the ordered list $\widetilde{{\mathcal{R}(G)^k}}$, which retains the ordering of ${\mathcal{R}(G)^k}$, but removes all non-representatives, so that the entries of $\widetilde{{\mathcal{R}(G)^k}}$ form a minimal generating set of $I(G)^k$.  We show in Theorem \ref{mainthm} that $\widetilde{{\mathcal{R}(G)^k}}$ is a linear quotient ordering.
\end{constr}

We illustrate this construction in the following example.

\begin{ex}
Returning to Example \ref{cycle example}, we have  
\begin{align*} 
{\mathcal{R}(\mathcal{C}_4)^2}
&= \left(m_1^2, m_1m_2, m_2^2, m_1m_3, m_2m_3, m_3^2, m_1m_4, m_2m_4, m_3m_4, m_4^2\right) \\ 
&= \left(x_1^2x_2^2, x_1x_2^2x_3, x_2^2x_3^2, x_1x_2x_3x_4, x_2x_3^2x_4, x_3^2x_4^2, x_1^2x_2x_4, x_1x_2x_3x_4, x_1x_3x_4^2, x_1^2x_4^2\right).
\end{align*}

Since $m_1m_3=x_1x_2x_3x_4$ and $m_2m_4=x_1x_2x_3x_4$ have the same vertex decompositions (and there are no other repetitions), we have that all elements of ${\mathcal{R}(\mathcal{C}_4)^2}$ are representatives except for $m_2m_4$.  Thus, according to our construction we remove $m_2m_4$ to get
\begin{displaymath}
    \widetilde{{\mathcal{R}(\mathcal{C}_4)^2}}=\left(x_1^2x_2^2, x_1x_2^2x_3, x_2^2x_3^2, x_1x_2x_3x_4, x_2x_3^2x_4, x_3^2x_4^2, x_1^2x_2x_4, x_1x_3x_4^2, x_1^2x_4^2\right).
\end{displaymath}
\end{ex}

The next example demonstrates that the lexicographic ordering of $\mathcal{S}(G)^k$ with respect to edge decompositions does not, in general, yield a linear quotient ordering.
 
\begin{ex}\label{counterex}  Consider the edge ideal $I(G)$ of the graph $G$ pictured below.
\begin{center}
\begin{tikzpicture}[scale = 0.75]
\node (a) at (2, 0) {$x_1$};
\node (b) at (1, {sqrt(3)}) {$x_2$};
\node (c) at (-1, {sqrt(3)}) {$x_3$};
\node (d) at (-2, 0) {$x_4$};
\node (e) at (-1, {-sqrt(3)}) {$x_5$};
\node (f) at (1, {-sqrt(3)}) {$x_6$};

\path [thick]
(a) edge node {} (c)
(a) edge node {} (b)
(c) edge node {} (d)
(a) edge node {} (e)
(c) edge node {} (e)
(e) edge node {} (f);
\end{tikzpicture}
\end{center}
It is easy to see that the following is a linear quotient ordering of $I(G)$:
\begin{align*}
   I(G)=\{m_1=x_1x_3, m_2=x_1x_2, m_3=x_3x_4, m_4=x_1x_5, m_5=x_3x_5, m_6=x_5x_6\}.
\end{align*}
The lexicographic ordering of the elements of $\mathcal{S}(G)^2$ forms a minimal generating set of $I(G)^2$ (as $\mathcal{S}(G)^2$ has no repetitions) and is given by
\begin{align*}
    I(G)^2=(m_1^2, m_1m_2, m_1m_3, m_1m_4, m_1m_5, m_1m_6, m_2^2,\dots,m_6^2).
\end{align*}
Observe that this is not a linear quotient ordering. Indeed, $m_1m_6=x_1x_3x_5x_6$ precedes $m_2m_3=x_1x_2x_3x_4$, but by quick inspection, the monomials $M$ of degree four such that $\frac{M}{\mathrm{gcd}(M, m_2m_3)}$ is linear and divides $\frac{m_1m_6}{\mathrm{gcd}(m_1m_6, m_2m_3)}=x_5x_6$ are precisely:
\begin{displaymath}
x_1x_2x_3x_5,\, x_1x_2x_4x_5,\, x_1x_3x_4x_5,\, x_2x_3x_4x_5,\, x_1x_2x_3x_6,\, x_1x_2x_4x_6,\, x_1x_3x_4x_6,\, x_2x_3x_4x_6,
\end{displaymath} 
and none of these precedes $m_2m_3=x_1x_2x_3x_4$ under the lexicographic ordering.
\end{ex}
 
\begin{remk}\label{counterrem}
    Observe that Example \ref{counterex} is a counterexample to \cite[Conjecture 4.1]{EFHHKM}, which states that if $I(G)^q$ has linear quotients for some $q\in\mathbb{N}$, then the resulting \textit{efficient ordering} (see \cite{EFHHKM} for the definition) is a linear quotients ordering of $I(G)^s$ for all $s\geq q$. 
 Indeed, in Example \ref{counterex}, the ordering on $I(G)^2$ coincides with the efficient ordering constructed from the ordering of $I(G)$.  Nevertheless, \cite[Conjecture 4.1]{EFHHKM} remains open for $q>1$. 
\end{remk}

We aim to prove that the list $\widetilde{{\mathcal{R}(G)^k}}$ in Construction \ref{ordering} yields a linear quotient ordering of $I(G)^k$; however, to simplify our proof, we reduce to working with
${\mathcal{R}(G)^k}$ instead.  For this, we need a simple lemma and some notation.

\begin{notat}\label{quotientnotation}

Given an ordered list $\mathcal{O}$ of monomials in $P$ and a monomial $M$ from the list, we denote by $Q^{\mathcal{O}}(M)$ the ideal quotient corresponding to the monomial $M$ under the ordering $\mathcal{O}$; that is, if $\mathcal{O}=(m_1,\dots,m_r)$ then
    \begin{align*}
        Q^{\mathcal{O}}(m_i)=(m_1,\dots,m_{i-1}):(m_i)
    \end{align*}
    for all $2\leq i\leq r$.  We suppress the superscript $\mathcal{O}$ and write $Q(M)$ to denote the ideal quotient corresponding to the monomial $M$ wherever the ordering is clear from context.

Note that for $M$ a non-representative in $\mathcal{R}(G)^k$, we have $Q(M)=P$.  We say that ${\mathcal{R}(G)^k}$ is a \textit{linear quotient ordering (up to repetition)} if for each generator $M$ in $\mathcal{R}(G)^k$, the ideal quotient $Q(M)$ is either generated by a subset of the variables of $P$ or is equal to the polynomial ring $P$.
\end{notat}

The next lemma follows directly from the definitions of $\widetilde{{\mathcal{R}(G)^k}}$ and ${\mathcal{R}(G)^k}$ in Construction \ref{ordering}.

\begin{lemma}\label{repetitions}
   Adopt notation in Construction \ref{ordering}.  Then $\widetilde{{\mathcal{R}(G)^k}}$ is a linear quotient ordering if and only if ${\mathcal{R}(G)^k}$ is a linear quotient ordering (up to repetition).
\end{lemma}


Now we state and prove two technical lemmas about our revlex ordering $\mathcal{R}(G)^k$ which we use in the proof of our main result in this section.  


\begin{lemma}\label{reduction-lemma}
Adopt notation in Construction \ref{ordering} and Notation \ref{quotientnotation}.  Let $m_i$ be a formal edge of $M:=m_1^{\alpha_1}\dotsc m_i^{\alpha_i}\dotsc m_r^{\alpha_r}\in{\mathcal{R}(G)^k}$.
Then for
$M':=\frac{1}{m_i} M\in{\mathcal{R}(G)^{k-1}}$ we have:
\begin{displaymath}
    Q^{\mathcal{R}(G)^{k-1}}(M') \subseteq Q^{\mathcal{R}(G)^{k}}(M).
\end{displaymath}
\end{lemma}

\begin{proof}
First note that the ideal quotient $Q^{\mathcal{R}(G)^{k-1}}(M')$ is generated by the monomials $\frac{M''}{\text{gcd}(M'',\,  M')}$ for all generators $M''$ which precede $M'$ in ${\mathcal{R}(G)^{k-1}}$. For each such $M''$,
it remains to show that the generator $m_i M''$ precedes $M$ in ${\mathcal{R}(G)^k}$, since then we have
\begin{align*}
    \frac{M''}{\text{gcd}(M'',\, M')} =\frac{m_i   M''}{\text{gcd}(m_i   M'',\, m_i   M')}=\frac{m_i   M''}{\text{gcd}(m_i   M'',\, M)}  \in Q^{\mathcal{R}(G)^{k}}(M).
\end{align*}  
But, indeed, this follows from the fact that $M''$ precedes $M'$ in the revlex ordering ${\mathcal{R}(G)^{k-1}}$.
\end{proof}

    

 \begin{lemma}\label{cases-lemma}
Adopt notation in Construction \ref{ordering} and suppose that 
\begin{displaymath}
    M=m_1^{\alpha_1} \dotsc\, m_r^{\alpha_r}=x_1^{a_1}\dotsc\, x_n^{a_n}\quad\text{precedes}\quad M'=~m_1^{\beta_1} \dotsc\, m_r^{\beta_r}=x_1^{b_1}\dotsc\, x_n^{b_n}
\end{displaymath}
in ${\mathcal{R}(G)^{k}}$, with $M'$ a representative.
    Let $p$ be the largest index such that $\alpha_p > 0$ and $q$ be the largest index such that $\beta_q > 0$.  Then at least one of the following is true:
    
    \begin{enumerate}
        \item [(a)] The elements $M$ and $M'$ share a formal edge $m_i$
        for some $i\in\{1,\dots,r\}$.
        \item [(b)]  
        There is a variable $x_i$ such that $b_i>a_i$ and a formal edge $m_j$ of $M'$ which is incident to $x_i$, such that for all formal edges $m_{\ell}$ of $M$, the element $\frac{1}{m_{\ell}} M$ precedes $\frac{1}{m_j} M'$ in  ${\mathcal{R}(G)^{k-1}}$. 
        \item [(c)] 
        The generator $\frac{1}{m_q} M'$ in  ${\mathcal{R}(G)^{k-1}}$ divides $M$ in $P$, with $p<q$.
    \end{enumerate}
\end{lemma}

\begin{proof}
    We prove that if the conditions (a) and (b) are false, then condition (c) must hold.  
    
    First note that under the revlex ordering, we have that $p \leq q$. Moreover, by our assumption that condition (a) is false, we have that $p < q$, as desired. 
    
    Since $M'$ is a representative, we have $M'\neq M$, and thus $b_i > a_i$, for some $i\in\{1,\dots,r\}$,  
    and in particular $x_i$ must divide some formal edge $m_j$ of $M'$.
    
    Now we claim that $\beta_q=1$, because if $\beta_q>1$, then condition (b) must hold. Indeed, $\beta_q>1$ implies that $m_q$ is a formal edge of $\frac{1}{m_j} M'$, but for any $m_{\ell}$ with $\alpha_{\ell}>0$,  we have that $\frac{1}{m_{\ell}} M$ only has formal edges with index at most $p$ and hence strictly less than $q$.  Thus   
    $\frac{1}{m_{\ell}} M$ precedes $\frac{1}{m_j} M'$ in ${\mathcal{R}(G)^{k-1}}$ under our revlex ordering, and thus the condition (b) holds.  
      
      In summary we have $p<q$ and $\beta_q=1$.  Now to show that condition (c) is true, it suffices to show that for all $i\in\{1,\dots,n\}$ and for all $j<q$ with $\beta_j>0$, we have that   $b_i \leq a_i$ whenever $x_i|\,m_j$.  For the sake of contradiction, we suppose that there exist indices $i\in\{1,\dots,n\}$ and $j<q$ with $\beta_j>0$ such that $b_i>a_i$ and $x_i|m_j$.  As before, we have that $\frac{1}{m_j} M'$ has the formal edge $m_q$, but $\frac{1}{m_{\ell}}M$ does not for any $m_{\ell}$ with $\alpha_{\ell}>0$.  Thus condition (b) must be true, which is a contradiction to our assumption. Therefore $\frac{1}{m_q} M'$ divides $M$, as desired.
\end{proof}

\begin{thm}\label{mainthm}
Let $P=\mathsf{k}[x_1,\dots,x_n]$ be a standard graded polynomial ring with $\mathsf{k}$ a field and let $I(G)$ be a quadratic monomial ideal that admits linear quotients. 
Further adopt notation in Construction \ref{ordering}. Then  $\widetilde{{\mathcal{R}(G)^{k}}}$ yields a linear quotient ordering on $I(G)^k$ for all $k\in\mathbb{N}$.
\end{thm}

\begin{proof}
By Lemma \ref{repetitions} it suffices to show that ${\mathcal{R}(G)^{k}}$ yields a linear quotient ordering (up to repetition). We use induction on the power $k$. The result holds for $k=1$ by hypothesis, since ${\mathcal{R}(G)^{1}}$ coincides with the linear quotient ordering of $I(G)$. 

Now let $k>1$ and assume ${\mathcal{R}(G)^{k-1}}$ is a linear quotient ordering (up to repetition). We need to show that ${\mathcal{R}(G)^{k}}$ also yields a linear quotient ordering (up to repetition). Suppose $M$ precedes $M'$ in ${\mathcal{R}(G)^{k}}$. Without loss of generality, we may assume that $M'$ is a representative; otherwise, $Q^{\mathcal{R}(G)^k}(M')=P$, as desired.
Now we apply Lemma \ref{cases-lemma} and examine each of the three possibilities separately.  As in the lemma, we consider the edge and vertex decompositions 
\begin{displaymath}
    M=m_1^{\alpha_1} \dotsc\, m_r^{\alpha_r}=x_1^{a_1}\dotsc\, x_n^{a_n}\quad\text{and}\quad M'=m_1^{\beta_1} \dotsc\, m_r^{\beta_r}=x_1^{b_1}\dotsc\, x_n^{b_n}
\end{displaymath}
and let $p$ and $q$ be the largest indices such that $\alpha_p > 0$ and $\beta_q > 0$.\vspace{0.15cm}

\noindent\textbf{(a)}  First we assume that $M$ and $M'$ share a common formal edge $m_i$ and define 
\begin{align*}
    M'_{i}:=\frac{1}{m_i} M'\quad\text{and}\quad M_{i}:=\frac{1}{m_i} M.
\end{align*}
Notice $M'_{i}$ is a representative in $\mathcal{R}(G)^{k-1}$ because by Lemma \ref{reduction-lemma} we have $Q^{\mathcal{R}(G)^{k-1}}(M'_{i})\subseteq Q^{\mathcal{R}(G)^{k}}(M')$.  By the inductive hypothesis and Lemma \ref{HWfact}, there is some variable $x\in{Q^{\mathcal{R}(G)^{k-1}}(M'_{i}})$ such that 
    \begin{align*}
        x\,\Big|\,\frac{M_{i}}{\text{gcd}(M_{i}, M'_{i})}.
    \end{align*}
    Thus by Lemma \ref{reduction-lemma} we have $x\in Q^{\mathcal{R}(G)^{k-1}}(M'_{i})\subseteq Q^{\mathcal{R}(G)^{k}}(M')$. Moreover, note that  
    \begin{align*}
        x\,\Big|\,\frac{M_{i}}{\text{gcd}(M_{i}, M'_{i})}=\frac{m_i   M_{i}}{\text{gcd}(m_i   M_{i}, m_i   M'_{i})}=\frac{M}{\text{gcd}(M, M')}.
    \end{align*} 
    Now Lemma \ref{HWfact} yields the desired result.
   \vspace{0.15cm}
   
\noindent\textbf{(b)} 
    Next we assume there is a variable $x_i$ such that $b_i>a_i$ and a formal edge $m_j$ of $M'$ which is incident to $x_i$, such that for all formal edges $m_{\ell}$ of $M$, the generator $M_{\ell}:=\frac{1}{m_{\ell}} M$ precedes $M'_j:=\frac{1}{m_j} M'$ in  ${\mathcal{R}(G)^{k-1}}$.  We write $m_j=x_ix_g$ for some $g\in\{1,\dots,n\}$ (possibly $g=i$) and choose a formal edge $m_\ell$ of $M$ as follows.  If $x_g$ does not divide $M$, select $m_{\ell}$ arbitrarily; otherwise, let $m_\ell$ be any formal edge of $M$ incident to $x_g$. 
    
    By the same argument as in (a), it suffices to show that 
    \begin{align*}
        \frac{M_{\ell}}{\text{gcd}(M_{\ell}, M'_j)}\,\Big|\, \frac{M}{\text{gcd}(M, M')}.
    \end{align*}
    Equivalently, we need to show that ${\text{gcd}(M, M')}\,\vert\,\text{gcd}(M, m_{\ell}M'_j)$, but this holds since $b_i>a_i$ 
    and since $m_\ell$ is divisible by $x_g$ whenever $x_g$ divides $M$.\vspace{0.15cm}
    
\noindent\textbf{(c)} Otherwise by Lemma \ref{cases-lemma}, we have $p<q$ and $M'_q:=\frac{1}{m_q} M'$ divides $M$ in $P$.
 Thus the degree of $\frac{M}{\text{gcd}(M, M')}$ is at most 2, and further since $M'$ is a representative, it must be linear or quadratic. 
    
    If it is linear, by Lemma \ref{HWfact} we are done. So we may assume that 
    \begin{align*}
        \frac{M}{\text{gcd}(M, M')}=x_{c_1}x_{c_2},
    \end{align*}
    for some variables $x_{c_1},x_{c_2}\in P$; we do not require that $c_1$ and $c_2$ are distinct, nor that $\{x_{c_1},x_{c_2}\}$ is an edge of $G$.  By Lemma \ref{HWfact} it suffices to show that $x_{c_1}\in{Q^{\mathcal{R}(G)^k}(M')}$ or $x_{c_2}\in{Q^{\mathcal{R}(G)^k}(M')}$.
    
    Let $m_q=x_ix_j$,
    so that $M'=x_ix_jM'_{q}$ and $M=x_{c_1}x_{c_2}M'_{q}$.  Observe that $x_i,x_j\notin\{x_{c_1},x_{c_2}\}$, because otherwise $\frac{M}{\text{gcd}(M, M')}$ is not quadratic.  Since $x_{c_1}$ divides $M$, there must be some formal edge $m_\ell$ of $M$ incident to $x_{c_1}$. Let $m_\ell=x_{c_1}x_{c_3}$, where $c_3$ need not be distinct from $c_1$ or $c_2$. 
 Since $p<q$, we have that $m_\ell$ precedes $m_q$ in the linear quotient ordering of $I(G)$, and therefore there is an edge $m_{q'}$ preceding $m_q$ in the linear quotient ordering of $I(G)$ with 
    \begin{align*}
        m_{q'}\in\{x_ix_{c_1},x_jx_{c_1},x_ix_{c_3},x_jx_{c_3}\}.
    \end{align*}
   
    Since $q'<q$, it follows that $m_{q'}M'_{q}$ precedes $M'$ in ${\mathcal{R}(G)^{k}}$. Thus, if $m_{q'}\in\{x_ix_{c_1},x_jx_{c_1}\}$, we have
    \begin{align*}
        \frac{m_{q'}M'_{q}}{\text{gcd}(m_{q'}M'_{q}, M')}=x_{c_1}\in{Q^{\mathcal{R}(G)^k}(M')},
    \end{align*}
and we are done.  Thus we may assume that $m_{q'}\in\{x_ix_{c_3},x_jx_{c_3}\}$. 
    Note that $\frac{m_{q'}}{m_{\ell}}M$ precedes $M'$ in ${\mathcal{R}(G)^{k}}$ since $m_\ell$ is a formal edge of $M$ and $p,q'<q$.  Thus the next observation completes the proof
    \begin{align*}
        \frac{\frac{m_{q'}}{m_{\ell}}M}{\text{gcd}(\frac{m_{q'}}{m_{\ell}}M, M')}=x_{c_2}\in{Q^{\mathcal{R}(G)^k}(M')}.
    \end{align*}
\end{proof}

The following corollary, which follows immediately from Theorem \ref{mainthm}, recovers a natural corollary of \cite[Theorems 10.1.9 and 10.2.5]{HHbook}; see also \cite[Theorem 2.6]{dali15}.

\begin{cor}\label{recover}
    A quadratic monomial ideal $I(G)$ has linear quotients if and only if $I(G)^k$ has linear quotients for all $k\in\mathbb{N}$. 
\end{cor}

\section{Powers of edge ideals of whisker graphs}

In this section, we illustrate the utility of our linear quotient ordering in Construction \ref{ordering}.  In particular, 
we apply Theorem \ref{mainthm} to obtain explicit formulas for the projective dimension and the Betti numbers of the powers of the edge ideal of a \textit{whisker graph}; that is, a graph $\mathcal{W}_{r,\ell}$ with $r+\ell+2$ vertices and $r+\ell+1$ edges
of the following form:

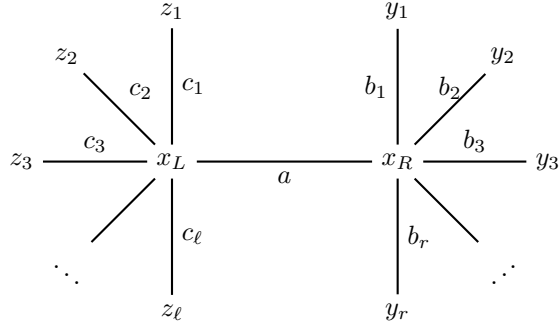
\begin{center}
\begin{tikzpicture}[scale = 1]
    \node (xL) at (0,0) {$x_L$};
    \node (xR) at (3,0) {$x_R$};
    \node (y1) at (0,2) {$z_1$};
    \node (y2) at (-1.4, 1.4) {$z_2$};
    \node (y3) at (-2, 0) {$z_3$};
    \node (e) at (-1.4, -1.4) {$\ddots$};
    \node (yn) at (0,-2) {$z_{\ell}$};
    \node (z1) at (3,2) {$y_1$};
    \node (z2) at (4.4, 1.4) {$y_2$};
    \node (z3) at (5,0) {$y_3$};
    \node (f) at (4.4, -1.4) {$\adots$};
    \node (zm) at (3,-2) {$y_r$};
    
    \draw [thick]
    (xL) edge node [below] {$a$} (xR)
         edge node [right] {$c_1$} (y1)
         edge node [above right] {$c_2$} (y2)
         edge node [above] {$c_3$} (y3)
         edge (e)
         edge node [right] {$c_\ell$} (yn)
    (xR) edge node [left] {$b_1$} (z1)
         edge node [above] {$b_2$} (z2)
         edge node [above] {$b_3$} (z3)
         edge (f)
         edge node [right] {$b_r$} (zm);
\end{tikzpicture}
\end{center}

Throughout this section, let $P=\mathsf{k}[x_R,x_L,y_1,\dots ,y_r,z_1,\dots ,z_{\ell}]$ be a standard graded polynomial ring over a field $\mathsf{k}$.  We begin by giving a linear quotient ordering on the edge ideal $I(\mathcal{W}_{r,\ell})$. 

\begin{prop}\label{whisker}
The ideal $I(\mathcal{W}_{r,\ell})$ has linear quotients with respect to the following ordering:
\begin{align*}
    I(\mathcal{W}_{r,\ell})=(a,b_1,b_2,\dots,b_r,c_1,c_2,\dots,c_{\ell}),
\end{align*}
where $a=\{x_R,x_L\}$, $b_i=\{x_R,y_i\}$ for $i\in\{1,2,...,r\}$, and $c_j=\{x_L,z_j\}$ for $j\in\{1,2,...,\ell\}$.
\end{prop}
\begin{proof}
   This follows directly from the following equalities for $i\in\{1,2,...,r\}$ and $j\in\{1,2,...,\ell\}$:
   \begin{align*}
       ((a,b_1,\dots,b_{i-1}):(b_i))&=\{x_L,y_1,\dots,y_{i-1}\}\\ ((a,b_1,\dots,b_r,c_1,\dots,c_{j-1}):(c_j))&=\{x_R,z_1,\dots,z_{j-1}\}.
   \end{align*}  
\end{proof}

Thus by Theorem \ref{mainthm}, $I(\mathcal{W}_{r,\ell})^k$ admits linear quotients for all $k\in\mathbb{N}$ under the revlex ordering $\widetilde{\mathcal{R}(\mathcal{W}_{r,\ell})^{k}}$
determined by the linear quotient ordering of $I(\mathcal{W}_{r,\ell})$ given in Proposition \ref{whisker}; see Construction \ref{ordering}.  We will adhere to this ordering throughout the remainder of this section.  Now we prove a general lemma that when $G$ is a tree, there are no repetitions in ${\mathcal{R}(G)^{k}}$; in particular, we have the equality $\widetilde{{\mathcal{R}(\mathcal{W}_{r,\ell})^{k}}}={\mathcal{R}(\mathcal{W}_{r,\ell})^{k}}$ since $\mathcal{W}_{r,\ell}$ is a tree. 

\begin{lemma}\label{tree}
Let $T$ be a tree on vertices $x_1,\dots,x_n$ and let $I(T)=(m_1, \dotsc, m_s)$ be its edge ideal in the standard graded polynomial ring $\mathsf{k}[x_1,\dots,x_n]$.  Then for all $k\in\mathbb{N}$, each generator $M \in I(T)^k$ can be expressed uniquely as $M = m_1^{\alpha_1} \dotsc m_s^{\alpha_s}$, for some list of nonnegative integers $\alpha=(\alpha_1, \dotsc, \alpha_s)$.
\end{lemma}

\begin{proof}
We proceed by (strong) induction on the number of vertices $n$. For the base case, consider the tree $T$ with two vertices, $x_1$ and $x_2$. Then $I(T)=(x_1x_2)$, and thus the only generator of $I(T)^k = (x_1^k x_2^k)$ can be represented uniquely, as desired. 
    
Now assume towards induction that the claim is true for any tree on fewer than $n$ vertices, and let $T$ be a tree on $n$ vertices. Then $T$ has at least one leaf; call this vertex $x_i$ and denote by $m_j$ the edge to which it is incident. 
Let $M$ be a minimal generator of $I(T)^k$. Then any power of $x_i$ dividing $M$ corresponds to a power of the edge $m_j$. Let $\alpha_j$ be the maximum power of $m_j$ dividing $M$ and let
    \begin{align*}
        M'=\frac{M}{m_j^{\alpha_j}}.
    \end{align*}
Then $M'$ is a generator of $I(T')^{k - \alpha_j}$, where $T'$ is the tree $T \setminus \{x_i \}$. By induction there is a unique list of nonnegative integers
     $(\alpha_1,\dotsc,\alpha_{j-1},\alpha_{j+1},\dotsc,\alpha_{s})$  such that $M'=m_1^{\alpha_1} \dotsc\, m_{j-1}^{\alpha_{j-1}}m_{j+1}^{\alpha_{j+1}}\dotsc\, m_s^{\alpha_s}$.  Now $M$ is expressed uniquely as $M = m_1^{\alpha_1} \dotsc\, m_s^{\alpha_s}$.
\end{proof}
 
Still adhering to our linear quotient ordering ${\mathcal{R}(\mathcal{W}_{r,\ell})^{k}}$ on $I(\mathcal{W}_{r,\ell})^k$, we now establish some additional notation to be used throughout the remainder of the section.  As in previous sections, we use the notation for ideal quotients introduced in Notation \ref{quotientnotation}, although in this section we write $Q(M)$ to denote the ideal quotient $Q^{\mathcal{R}(G)^k}(M)$ corresponding to the monomial $M$ in $\mathcal{R}(G)^k$ as this is the only ordering we will use throughout the section.  We also refer to edge and vertex decompositions of the generators of $I(\mathcal{W}_{r,\ell})^k$ as introduced in Notation \ref{decomp} throughout this section.  Note however in this section, by Lemma \ref{tree},
not only the vertex decomposition but also the edge decomposition of a generator of $I(\mathcal{W}_{r,\ell})^k$ is unique. 
 Thus there is no distinction between an edge and a formal edge. 
 Furthermore we introduce the following notation.
  
\begin{notat}\label{BCnot}
For any generator $M=a^{\alpha}b_1^{\beta_1}...b_r^{\beta_r}c_1^{\gamma_1}...c_\ell^{\gamma_\ell}$ of $I(\mathcal{W}_{r,\ell})^{k}$, we define integers
\begin{align*}
   B(M)=\begin{cases}
   \text{max}\{p:\beta_p>0\}, & \{p:\beta_p>0\}\neq\emptyset \\
0, & \text{otherwise}
   \end{cases}
   \quad\quad
    C(M)=\begin{cases}
   \text{max}\{p:\gamma_p>0\}, & \{p:\gamma_p>0\}\neq\emptyset \\
0, & \text{otherwise}
   \end{cases}.
\end{align*}
\end{notat}

  
To obtain formulas for the projective dimension and Betti numbers of $I(\mathcal{W}_{r,\ell})^k$, we need the following technical lemma which describes, for a given generator $M$ of $I(\mathcal{W}_{r,\ell})^k$, when each variable belongs to the corresponding ideal quotient $Q(M)$.

\begin{lemma}\label{whiskerquotients}
Adopt Notations \ref{quotientnotation} and \ref{BCnot}.  Fix an integer $k\geq 1$, and let $M=a^{\alpha}b_1^{\beta_1}...b_r^{\beta_r}c_1^{\gamma_1}...c_\ell^{\gamma_\ell}$ be a generator of $I(\mathcal{W}_{r,\ell})^{k}$.  The following statements hold:
\begin{itemize}
    \item[(1)]  For all $1\leq j\leq r$, $y_j\in Q(M)$ \,if and only if\, $B(M)>j$;
    \item[(2)]  For all $1\leq j\leq \ell$, $z_j\in Q(M)$ \,if and only if\, $C(M)>j$;
    \item[(3)]  $x_R\in Q(M)$ if and only if $C(M)>0$;
       \item[(4)]  $x_L\in Q(M)$ if and only if $B(M)>0$;
\end{itemize}

\end{lemma}

\begin{proof}
Consider
the unique vertex decomposition
\begin{align*}
    M={y_1^{d_1}}\cdot\dots\cdot{y_r^{d_r}}{z_1^{d_{r+1}}}\cdot\dots\cdot{z_\ell^{d_{r+\ell}}}x_R^{d_{r+\ell+1}}x_L^{d_{r+\ell+2}}.
\end{align*}
One can check from the structure of $\mathcal{W}_{r,\ell}$ that the following constraints hold:
\begin{itemize}
    \item  Constraint 1: $\begin{cases}
        d_i>0\,\iff\,\beta_i>0, & \text{for}\,\, 1\leq i\leq r \\
        d_i>0\,\iff\,\gamma_{i-r}>0, & \text{for}\,\, r+1\leq i\leq r+\ell
    \end{cases}$
    \item Constraint 2:\,\, $\displaystyle{d_{r+\ell+1}-\sum_{s=1}^{r} d_s = d_{r+\ell+2}-\sum_{s=1}^{\ell} d_{r+s}}$. 
\end{itemize}
Now we are ready to prove the desired statements.  For ease of notation, in the following arguments we abbreviate $B(M)$ and $C(M)$ to $B$ and $C$, respectively.  Note that the proofs of (2) and (4) are entirely similar to those of (1) and (3), respectively, so we omit those here.\vspace{0.15cm}

\noindent\textbf{(1)} Fix $1\leq j\leq r$ and assume that $B>j$. Then $b_B={x_Ry_B}$ is an edge of $M$, and thus $M\cdot\frac{b_j}{b_B}$ precedes $M$ in the linear quotient ordering. Thus we have:
\begin{align*}
   y_j=\frac{M\cdot\frac{x_Ry_j}{x_Ry_B}}{\mathrm{gcd}\left(M\cdot\frac{x_Ry_j}{x_Ry_B}, M\right)}=\frac{M\cdot\frac{b_j}{b_B}}{\mathrm{gcd}\left(M\cdot\frac{b_j}{b_B}, M\right)}\in Q(M). 
\end{align*}

Conversely if $y_j\in Q(M)$, then $(M'):(M)=(y_j)$ for some $M'$ preceding $M$; that is: 
\begin{align}\label{whiskervertexdecomp}
M'={y_1^{e_1}}\cdot\dots\cdot{y_r^{e_r}}{z_1^{e_{r+1}}}\cdot\dots\cdot{z_\ell^{e_{r+\ell}}}x_R^{e_{r+\ell+1}}x_L^{e_{r+\ell+2}}, 
\end{align}
where for some $w\neq j$, we have equalities $d_n=e_n$ for all $n\in\{1,\dots,r+\ell+2\}{\setminus\{j,w\}}$, $e_j=d_j+1$, and $e_w=d_w-1$.  By double application of Constraint 2, we have the following equalities
 \begin{align}\label{whiskerconstraint2}
     d_{r+\ell+1}-\sum_{s=1}^{r} d_s &= d_{r+\ell+2}-\sum_{s=1}^{\ell} d_{r+s}\nonumber \\ e_{r+\ell+1}-\sum_{s=1}^{r} e_s &= e_{r+\ell+2}-\sum_{s=1}^{\ell} e_{r+s}.
 \end{align}
 Since $e_j=d_j+1$, $e_w=d_w-1$, and $1\leq j\leq r$, it follows from these equations that $1\leq w\leq r$ or $w=r+\ell+2$.  We consider these two cases separately.
 
 If $1\leq w\leq r$ then we have the following equality 
 \begin{align*}
     M'\cdot{x_Ry_w}=M\cdot{x_Ry_j}.
 \end{align*}
Now by uniqueness of edge decompositions of generators of $I(\mathcal{W}_{r,\ell})^{k+1}$
 and since $M'$ precedes $M$ in the linear quotient ordering of $I(\mathcal{W}_{r,\ell})^k$, we must have that $x_Ry_j$ precedes $x_Ry_w$ in the linear quotient ordering of $I(\mathcal{W}_{r,\ell})$. Thus, $w>j$. Since $e_w=d_w-1$, we have $d_w>0$, and thus by Constraint 1 we have $\beta_w>0$. Therefore $B\geq{w}>j$, as desired. 
 
 On the other hand, if $w=r+\ell+2$, then we have the following equality
\begin{align*}
    M'\cdot{x_Rx_L}=M\cdot{x_Ry_j},
\end{align*}
but this is impossible because $M'$ precedes $M$ in the linear quotient ordering of $I(\mathcal{W}_{r,\ell})^k$ and $x_Rx_L$ precedes $x_Ry_j$ in the linear quotient ordering of $I(\mathcal{W}_{r,\ell})$. \vspace{0.15cm}
\noindent \textbf{(3)} 
Assume that $C>0$. Then $c_C={x_Lz_C}$ is an edge of $M$, and thus $M\cdot\frac{a}{c_C}$ precedes $M$, yielding:
\begin{align*}
x_R=\frac{M\cdot\frac{x_Rx_L}{x_Lz_C}}{\mathrm{gcd}\left(M\cdot\frac{x_Rx_L}{x_Lz_C}, M\right)}=\frac{M\cdot\frac{a}{c_C}}{\mathrm{gcd}\left(M\cdot\frac{a}{c_C}, M\right)}\in Q(M).  
\end{align*}

Conversely, if $x_R\in Q(M)$ then $(M'):(M)=(x_R)$ for some $M'$ preceding $M$. Thus $M'$ has the vertex decomposition in \eqref{whiskervertexdecomp}, 
where for some $w\neq r+\ell+1$, we have equalities $d_n=e_n$ for all $n\in\{1,\dots,r+\ell+2\}{\setminus\{r+\ell+1,w\}}$, $e_{r+\ell+1}=d_{r+\ell+1}+1$, and $e_w=d_w-1$.  Again by double application of Constraint 2, we have the equations in \eqref{whiskerconstraint2}.
Since $e_{r+\ell+1}=d_{r+\ell+1}+1$ and $e_w=d_w-1$, it follows from these equations that $r+1\leq w\leq r+\ell$. Since $e_w=d_w-1$, we have $d_w>0$, and thus by Constraint 1 we have $\gamma_{w-r}>0$. Therefore, $C\geq{w-r}>0$, as desired.
\end{proof}

Now we are ready to obtain formulas for the projective dimension and the Betti numbers of the powers of the edge ideal of the whisker graph $\mathcal{W}_{r,\ell}$. 

\begin{thm}\label{pd&betti-whisker}
Let $P=\mathsf{k}[x_R,x_L,y_1,\dots ,y_r,z_1,\dots ,z_{\ell}]$ be a standard graded polynomial ring with $\mathsf{k}$ a field.  Then the projective dimension of the ideal $I(\mathcal{W}_{r,\ell})^k$ is given by
\begin{align*}
    \mathrm{pd}(I(\mathcal{W}_{r,\ell})^k)=\begin{cases} r+\ell & k\geq 2 \\
    \mathrm{max}(r,\ell) & k=1
    \end{cases}
\end{align*}
and for $k\geq 1$ its Betti numbers are given by
\begin{displaymath}
\beta_i(I(\mathcal{W}_{r,\ell})^k)=\begin{dcases} {k-2+i\choose i}{r+\ell+k\choose k+i}+{k-2+i\choose i-1}\left[{r+k\choose k+i}+{\ell+k\choose k+i}\right] & i\geq 1 \\
    {r+\ell+k\choose k} & i=0.
    \end{dcases}
\end{displaymath}
\end{thm}

\begin{proof}
By Theorem \ref{mainthm} and Proposition \ref{whisker}, the powers of $I(\mathcal{W}_{r,\ell})$ have linear quotients with respect to the revlex ordering
$\mathcal{R}({\mathcal{W}_{r,\ell})^k}=\left(M_1,M_2,...,M_{\binom{r+\ell+k}{k}}\right)$
in Construction \ref{ordering}. Thus by Proposition \ref{pd&betti}, to calculate projective dimension and Betti numbers, it suffices to calculate the minimal number of generators $\nu(Q(M_j))$ of the ideal quotient $Q(M_j)$, for each $j=1,\dots,{r+\ell+k\choose k}$.
 
By Lemma \ref{whiskerquotients}, we have the following equality for each $j=1,\dots,{r+\ell+k\choose k}$,
\begin{align}\label{countgens}
    \nu(Q(M_j))=B(M_j)+C(M_j),
\end{align}
where $B(M_j)$ and $C(M_j)$ are defined as in Notation \ref{BCnot}.  Thus $\nu(Q(M_j))$ is maximized whenever $B(M_j)$ and $C(M_j)$ are as large as possible.  For $k\geq 2$, this is achieved when $B(M_j)=r$ and $C(M_j)=\ell$, and thus by Proposition \ref{pd&betti} the projective dimension is $r+\ell$. For $k=1$, it is impossible for both $B(M_j)$ and $C(M_j)$ to be positive, so in this case the projective dimension is $\mathrm{max}(r,\ell)$.  

Now we calculate the Betti numbers $\beta_i:=\beta_i(I(\mathcal{W}_{r,\ell})^k)$.  By Proposition \ref{pd&betti} we have:
\begin{align*}
\beta_0=\sum_{j=1}^{r+\ell+k\choose k}{\nu(Q(M_j))\choose 0}=\sum_{j=1}^{r+\ell+k\choose k}1={r+\ell+k\choose k}.  
\end{align*}

Next, we calculate $\beta_i$ for $i>0$. For fixed integers $B,C\geq 0$, define 
\begin{align}\label{ffunction}
    f(B,C):=\left|\left\{M\in{\mathcal{R}({\mathcal{W}_{r,\ell})^k}}\,\Big|\,B(M)=B\,\,\text{and}\,\,C(M)=C\right\}\right|.
\end{align}
Then by Proposition \ref{pd&betti} and \eqref{countgens} we have the equality 
\begin{align}\label{bettiformulaBC}
\beta_i =\sum_{B,C\geq 1}f(B,C){B+C\choose i}+\sum_{\substack{B\geq 1\\C=0}}f(B,0){B \choose i}+\sum_{\substack{B=0\\C\geq 1}}f(0,C){C\choose i}.
\end{align}

Observe that $f(B,0)$ is the number of sequences of length $B+1$ of nonnegative integers which sum to $k$ and whose last entry is strictly positive; that is,
\begin{align}\label{countciszero}
    f(B,0)={B+k\choose k}-{B+k-1\choose k}={B+k-1\choose k-1}={B+k-1\choose B},
\end{align}
where the second equality follows from Pascal's Identity.  Similarly we have
\begin{align}\label{countbiszero}
    f(0,C)={C+k-1\choose C}.
\end{align}
Now for ${B,C\geq 1}$, $f(B,C)$ is the number of sequences of length $B+C+1$ of nonnegative integers which sum to $k$ whose last two entries (corresponding to the exponents of $b_B=x_Ry_B$ and $c_C=x_Lz_C$) are strictly positive. By the inclusion-exclusion principle, we see that
\begin{align}
     f(B,C)&={B+C+k\choose k}-2{B+C+k-1\choose k}+{B+C+k-2\choose k} 
    ={B+C+k-2\choose B+C},\label{counttherest}
\end{align}
where the last equality follows from several applications of Pascal's identity and the symmetry of binomial coefficients.

Substituting \eqref{countciszero}, \eqref{countbiszero}, and \eqref{counttherest} into \eqref{bettiformulaBC} and writing $A:=B+C$, we obtain
\begin{align}\label{bettiformulasimplified}
\beta_i &=\sum_{B,C\geq 1}{A+k-2\choose A}{A\choose i}+\sum_{\substack{B\geq 1\\C=0}}{B+k-1\choose B}{B \choose i}+\sum_{\substack{B=0\\C\geq 1}}{C+k-1\choose C}{C\choose i}\nonumber\\
  &=\sum_{B,C\geq 1}{A+k-2\choose k-2+i}{k-2+i\choose k-2}+\sum_{\substack{B\geq 1\\C=0}}{B+k-1\choose k-1+i}{k-1+i \choose k-1}+\sum_{\substack{B=0\\C\geq 1}}{C+k-1\choose k-1+i}{k-1+i\choose k-1} \nonumber\\
  &={k-2+i\choose k-2}\sum_{B,C\geq 1}{A+k-2\choose k-2+i}+{k-1+i\choose k-1}\left[\sum_{\substack{B\geq 1\\C=0}}{B+k-1\choose k-1+i}+\sum_{\substack{B=0\\C\geq 1}}{C+k-1\choose k-1+i}\right],\nonumber\\
  \,
\end{align}
\vspace*{-1cm}\\
where the second equality follows from the combinatorial identity ${x\choose y}{y\choose z}={x\choose x-y+z}{x-y+z\choose x-y}$.

Now we analyze each sum in the equality above separately.  By the hockey-stick identity, since $B$ runs from 1 to $r$, we have the following equality
\begin{align}\label{hockey1}
    \sum_{\substack{B\geq 1\\C=0}}{B+k-1\choose k-1+i}=\sum_{B=i}^r{B+k-1\choose k-1+i}={r+k\choose k+i}.
\end{align}
Similarly since $C$ runs from 1 to $\ell$, we have the following equality
\begin{align}\label{hockey2}
    \sum_{\substack{B=0\\C\geq 1}}{C+k-1\choose k-1+i}=\sum_{C=i}^{\ell}{C+k-1\choose k-1+i}={\ell+k\choose k+i}.
\end{align}
We use the same identity twice on the remaining sum as follows:
\begin{align}\label{hockey3}
\sum_{B,C\geq 1}{A+k-2\choose k-2+i}&=\sum_{B\geq 1}\sum_{C\geq 1}{B+C+k-2\choose k-2+i} \nonumber=\sum_{B\geq 1}\left[{B+\ell+k-1\choose k-1+i}-{B+k-1\choose k-1+i}\right] \nonumber\\
    &=\left[{r+\ell+k\choose k+i}-{\ell+k\choose k+i}\right]-\left[{r+k\choose k+i}-{k\choose k+i}\right] \nonumber\\
    &={r+\ell+k\choose k+i}-{r+k\choose k+i}-{\ell+k\choose k+i}.
\end{align}

Substituting \eqref{hockey1}, \eqref{hockey2}, and \eqref{hockey3} into \eqref{bettiformulasimplified} and simplifying, we have
\begin{align*}
    \beta_i &={k-2+i\choose k-2}\left[{r+\ell+k\choose k+i}-{r+k\choose k+i}-{\ell+k\choose k+i}\right]+{k-1+i\choose k-1}\left[{r+k\choose k+i}+{\ell+k\choose k+i}\right]\\
&={k-2+i\choose k-2}{r+\ell+k\choose k+i}+\left[{k-1+i\choose k-1}-{k-2+i\choose k-2}\right]
\left[{r+k\choose k+i}+{\ell+k\choose k+i}\right] \\
    &={k-2+i\choose i}{r+\ell+k\choose k+i}+{k-2+i\choose i-1}\left[{r+k\choose k+i}+{\ell+k\choose k+i}\right],
\end{align*}
where the last equality follows from Pascal's Identity and the symmetry of binomial coefficients.  
\end{proof}

This result generalizes the formulas given by Ferra, Murgia, and Olteanu in \cite[Corollary 3.4, Remark 3.5]{FMO12} for so-called \textit{star graphs}; that is, whisker graphs with $\ell=0$.

\section{Powers of the edge ideal of the anticycle}

In this section we provide a linear quotient ordering on the second and higher powers of the edge ideal of the anticycle graph (that is, the complement of the simple cycle), despite the fact that the edge ideal of the anticycle itself does not admit linear quotients. This answers in the affirmative a question of Hoefel and Whieldon; see \cite[Question 5.1]{HW11}.

We begin by establishing notation we will use throughout the section. 
Fix integers $n\geq 5$ and $k\geq 2$ and let $P=\mathsf{k}[x_1,\dots,x_n]$.  Denote by $\mathcal{A}_n$ the anticycle on vertices $x_1,\dots,x_n$ with edge ideal 
\begin{displaymath}
    I(\mathcal{A}_n)=(x_1x_3, x_1x_4,\dots, x_1x_{n-1}, x_2x_4,\dots, x_2x_n,\dots, x_{n-2}x_n).
\end{displaymath}
Note that $I(\mathcal{A}_n)$ is obtained from the edge ideal of the antipath on the same $n$ vertices,
\begin{displaymath}
    I(\mathcal{P}_n)=(x_1x_3, x_1x_4, \dots, x_1x_{n}, x_2x_4,\dots ,x_2x_n,\dots, x_{n-2}x_n)
\end{displaymath}
by simply removing the edge $x_1x_n$.  As in previous sections, we refer to edge and vertex decompositions of the generators of $I(\mathcal{A}_n)$ as introduced in Terminology \ref{decomp} as well as the ideal quotients $Q^{\mathcal{O}}(M)$ of a monomial $M$ under the ordering $\mathcal{O}$ as introduced in Notation \ref{quotientnotation} throughout this section.  Furthermore, for a monomial $M\in P$, we denote by $\mathrm{supp}\,M$ the usual support of a monomial; that is, the set of variables appearing with positive exponent in the vertex decomposition of $M$. 

  
Next we construct what we prove in Theorem \ref{anticycle-theorem} to be a linear quotient ordering on $I(\mathcal{A}_n)^k$.

\begin{constr}\label{anticycle-order}
    Denote by $\mathcal{O}_n^{(k)}$ the following ordering of the minimal generators of $I(\mathcal{A}_n)^k$:

\begin{itemize}
    \item[(1)]  First, order all the generators divisible by $x_n$ according to the lexicographic ordering by: 
    \begin{displaymath}
        x_n>x_2>x_3>\cdots >x_{n-1}>x_1;
    \end{displaymath}
    that is, $x_n^{\alpha_n}x_2^{\alpha_2}x_3^{\alpha_3}\dots x_{n-1}^{\alpha_{n-1}}x_1^{\alpha_1}$ precedes $x_n^{\beta_n}x_2^{\beta_2}x_3^{\beta_3}\dots x_{n-1}^{\beta_{n-1}}x_1^{\beta_1}$ whenever the first nonzero entry in the vector $(\alpha_n,\alpha_2,\alpha_3,\dots,\alpha_{n-1},\alpha_1)-(\beta_n,\beta_2,\beta_3,\dots,\beta_{n-1},\beta_1)$ is positive.
    \item[(2)] Next, order the remaining generators according to the lexicographic ordering by
    \begin{displaymath}
        x_1>x_2>x_3>\cdots >x_{n-1};
    \end{displaymath}
    that is, $x_1^{\alpha_1}x_2^{\alpha_2}\dots x_{n-1}^{\alpha_{n-1}}$ precedes $x_1^{\beta_1}x_2^{\beta_2}\dots x_{n-1}^{\beta_{n-1}}$ whenever the first nonzero entry in the vector $(\alpha_1,\alpha_2,\dots,\alpha_{n-1})-(\beta_1,\beta_2,\dots,\beta_{n-1})$ is positive.
    \item[(3)] Finally, move the generator $(x_1x_{n-1})^k$ immediately after the generator 
    \begin{align*}
        D:=(x_1x_{n-1})^{k-1}(x_2x_{n-1}).
    \end{align*}
    We call $D$ the \textit{distinguished generator}.
\end{itemize}

Note that $\mathcal{O}_n^{(k)}$ is a concatenation of two orderings $F$ and $S$, where $F$ is the sub-ordering containing all generators divisible by $x_n$ as described in (1) and $S$ is the sub-ordering containing all generators not divisible by $x_n$ as described in (2) and (3).
\end{constr}

\begin{ex}
     Let $n=5$ and $k=2$.  The ordering constructed above is $\mathcal{O}_n^{(k)}=(F,S)$, where
     \begin{align*}
     F&=(x_2^2x_5^2,\, x_2x_3x_5^2,\, x_3^2x_5^2,\, x_2^2x_4x_5,\, x_2x_3x_4x_5,\, x_1x_2x_3x_5,\, x_1x_2x_4x_5,\, x_1x_3^2x_5,\, x_1x_3x_4x_5) \\
     S&=(x_1^2x_3^2,\, x_1^2x_3x_4,\, x_1x_2x_3x_4,\, x_1x_2x_4^2,\, x_1^2x_4^2,\, x_2^2x_4^2). 
     \end{align*}
\end{ex}

\begin{remk}
    Note that the concatenation of two linear quotient orderings need not be a linear quotient ordering. For instance, $(x_1x_2, x_2x_3)$ and $(x_3x_4, x_4x_5, x_1x_5)$ are both linear quotient orderings, but their concatenation is not; indeed, the simple cycle $\mathcal{C}_5$ admits no linear quotient ordering.
    
    On the other hand, the concatenation of two orderings being a linear quotient ordering does not imply that the second ordering is itself a linear quotient ordering. For example, observe that $(x_2x_3, x_1x_2, x_3x_4)$ is a linear quotient ordering of $\mathcal{P}_4$, but $(x_1x_2, x_3x_4)$ does not yield linear quotients. 
    
    Surprisingly, however, we show in Theorem \ref{anticycle-theorem} that $F$ and $S$ as defined in Construction \ref{anticycle-order} interact ``nicely" with each other, in that their concatenation is a linear quotient ordering and furthermore they are both themselves linear quotient orderings.
\end{remk}

Next we illustrate why some natural simplifications of $\mathcal{O}_n^{(k)}$ do not yield linear quotient orderings.

\begin{remk}
\textbf{(1)}  The lexicographic ordering of all of the generators by 
\begin{displaymath}
    x_n>x_2>x_3>\cdots >x_{n-1}>x_1
\end{displaymath}
does not yield a linear quotient ordering of $I(\mathcal{A}_n)^k$ because in this case if $M_1=x_nx_1x_3^2(x_1x_3)^{k-2}$ and $M_2=x_n^2x_2^2(x_1x_3)^{k-2}$, there is no generator $M_3$ preceding $M_2$ such that

\begin{displaymath}
    \frac{M_3}{\text{gcd}(M_3, M_2)}\in\{x_1,x_2\dots , x_n\}\quad\text{and}\quad\frac{M_3}{\text{gcd}(M_3, M_2)}\Big|\frac{M_1}{\text{gcd}(M_1, M_2)}.
\end{displaymath} 
This illustrates the importance of our choice in Construction \ref{anticycle-order}(2). \\
\noindent\textbf{(2)}  The lexicographic ordering of the generators divisible by $x_n$ by
\begin{displaymath}
    x_n>x_2>x_3>\cdots >x_{n-1}>x_1
\end{displaymath}
followed by the lexicographic ordering of the remaining generators by
\begin{displaymath}
    x_1>x_2>x_3>\cdots >x_{n-1}
\end{displaymath}
does not yield a linear quotient ordering of $I(\mathcal{A}_n)^k$, because in this case if $M_1=x_2x_{n}(x_1x_{n-1})^{k-1}$ and $M_2=(x_1x_{n-1})^k$, there is no generator $M_3$ preceding $M_2$ such that 
\begin{displaymath}
    \frac{M_3}{\text{gcd}(M_3, M_2)}\in\{x_1,x_2\dots , x_n\}\quad\text{and}\quad\frac{M_3}{\text{gcd}(M_3, M_2)}\Big|\frac{M_1}{\text{gcd}(M_1, M_2)}. 
\end{displaymath}
This illustrates the final move in Construction 
\ref{anticycle-order} (3), which surprisingly requires no further changes.
\end{remk}

\begin{thm}\label{anticycle-theorem}
Let $P=k[x_1,\dots, x_n]$ be a standard graded polynomial ring with $n\geq 5$, and let $I(\mathcal{A}_n)\subseteq P$ be the edge ideal of the anticycle graph $\mathcal{A}_n$ on $n$ vertices.  Then for all integers $k\geq 2$, the ordering $\mathcal{O}_n^{(k)}$ defined in Construction \ref{anticycle-order} is a linear quotient ordering of $I(\mathcal{A}_n)^k$.
\end{thm} 

\begin{proof}
Let $M_1$ and $M_2$ be part of a minimal generating set for $I(\mathcal{A}_n)^k$, with $M_1$ preceding $M_2$ in $\mathcal{O}_n^{(k)}$.  Assume the following vertex decompositions:
\begin{align*}
    M_1=x_1^{\alpha_1}\dots x_n^{\alpha_n}\quad\text{and}\quad M_2=x_1^{\beta_1}\dots x_n^{\beta_n},
\end{align*} 
and define vectors $\alpha=(\alpha_1,\dots,\alpha_n)$ and $\beta=(\beta_1,\dots,\beta_n)$.  By Lemma \ref{HWfact}, it suffices to show that there exists some $M_3$ preceding $M_2$ such that 
\begin{displaymath}
    \frac{M_3}{\text{gcd}(M_3, M_2)}\in\{x_1,x_2\dots , x_n\}\quad\text{and}\quad \frac{M_3}{\text{gcd}(M_3, M_2)}\Big|\frac{M_1}{\text{gcd}(M_1, M_2)}.
\end{displaymath}
We separate our proof into three main cases depending on whether $M_1$ and $M_2$ are in $F$ or $S$.  Note that since $M_1$ precedes $M_2$ in $\mathcal{O}_n^{(k)}$, the case where $M_1$ is in $S$ and $M_2$ is in $F$ is impossible.\vspace{0.15cm}

\noindent\textbf{Case 1:} In this case we assume that $M_1\in{F}$ and $M_2\in{S}$, and proceed with two subcases.\vspace{0.15cm}

\noindent\textbf{Subcase 1.1:}  First we assume that $M_2=(x_1x_{n-1})^k$. Since $M_1\in{F}$, there is some $i\in\{2,\dots, n-2\}$ such that $x_ix_n$ is an edge of $M_1$.  We choose a generator $M_3$ of $I(\mathcal{A}_n)^k$ as follows:
\begin{displaymath}
M_3:=
    \begin{cases}
    M_2\cdot\frac{x_ix_{n-1}}{x_1x_{n-1}}, & \text{if}\,\,\, i=2 \\
    M_2\cdot\frac{x_1x_i}{x_1x_{n-1}}, & \text{if}\,\,\, i>2 
\end{cases}
\end{displaymath}
In either case, $M_3$ precedes $M_2$ in $\mathcal{O}_n^{(k)}$.  Indeed, when $i>2$ this follows from the lex ordering on $S$ in Construction \ref{anticycle-order}(2), and otherwise $M_3$ is the distinguished generator and this follows from Construction \ref{anticycle-order}(3).  Furthermore we have:
\begin{align*}
\frac{M_3}{\text{gcd}(M_3, M_2)}=x_i\quad\text{and}\quad x_i\Big|\frac{M_1}{\text{gcd}(M_1, M_2)},
\end{align*}
where the equality follows directly from the definitions of $M_2$ and $M_3$ and the divisibility statement follows from the fact that $x_i$ divides $M_1$, but not $M_2$.  This completes this subcase.\vspace{0.15cm} 

\noindent\textbf{Subcase 1.2:} Now we assume that $M_2\neq(x_1x_{n-1})^k$. Since $M_1\in{F}$ and $M_2\in{S}$, it follows that $x_n$ divides $M_1$ but not $M_2$, and thus
\begin{displaymath}
    x_n\Big|\frac{M_1}{\text{gcd}(M_1, M_2)}.
\end{displaymath}
Now it suffices to find a generator $M_3$ preceding $M_2$ in $\mathcal{O}_n^{(k)}$ such that 
\begin{align}\label{1.2NTS}
    \frac{M_3}{\text{gcd}(M_3, M_2)}=x_n.
\end{align}
By assumption, $M_2$ must have an edge $x_ax_b$ such that $1\leq{a}<b\leq{n-1}$ and $x_ax_b\neq{x_1x_{n-1}}$.  In particular, if $a=1$, then $1<b<n-1$, and thus $x_nx_b$ is an edge of $\mathcal{A}_n$ in this case.  Therefore, we may choose a generator $M_3$ of $I(\mathcal{A}_n)^k$ as follows:
\begin{displaymath}
M_3:=
    \begin{cases}
    M_2\cdot\frac{x_nx_b}{x_ax_b}, & \text{if}\,\,\, a=1 \\
    M_2\cdot\frac{x_ax_n}{x_ax_b}, & \text{if}\,\,\, a>1 
\end{cases}
\end{displaymath}
In either case, it follows from Construction \ref{anticycle-order}(1) that $M_3$ precedes $M_2$ in $\mathcal{O}_n^{(k)}$ and satisfies \eqref{1.2NTS}.\vspace{0.15cm}

\noindent\textbf{Case 2:} In this case we assume that $M_1,M_2\in{S}$.
Note that it suffices to show $S$ is itself a linear quotient ordering. Since $S$ is an ordering of the minimal generators of $I(\mathcal{A}_n)^k$ not divisible by $x_n$, it follows that $S$ is an ordering of the minimal generators of $I(P_{n-1})^k$. Furthermore, by \cite[Proposition 3.1]{HW11}, the lex ordering $x_1>x_2>x_3>\dots>x_{n-1}$, which we denote by $L$, yields a linear quotient ordering on $I(P_{n-1})^k$.  Thus it remains only to show that our choice to move $(x_1x_{n-1})^k$ immediately after the distinguished generator as described in Construction \ref{anticycle-order}(3) does not disturb the linear quotient property.

Observe that if $M_2$ precedes $(x_1x_{n-1})^k$ in $L$ or if $M_2$ succeeds the distinguished generator $D:=(x_1x_{n-1})^{k-1}(x_2x_{n-1})$ in $L$, then $Q^L(M_2)=Q^S(M_2)$. 
Thus, since $L$ is a linear quotient ordering, we may assume without loss of generality that $M_2$ lies between $(x_1x_{n-1})^k$ and $D$ (inclusive) in $L$.  We proceed with three subcases.\vspace{0.15cm}

\noindent\textbf{Subcase 2.1:} First we assume that $M_2$ lies strictly between $(x_1x_{n-1})^k$ and $D$ in $L$.  Thus by the lex ordering $L$, we have that $\beta_1=k-1$, and so $x_1^{k-1}$ divides $M_2$.  Further since $M_2$ precedes $D$ in $L$, we have $\beta_2\geq 1$, implying that $x_2$ divides $M_2$.  Thus $x_1^{k-1}x_2$ divides $M_2$, while $x_1^{k}$ cannot divide $M_2$. It follows that
\begin{align*}
    \frac{M_2}{\text{gcd}(M_2, (x_1x_{n-1})^k)}\notin\{x_1,x_2\dots , x_n\}.
\end{align*}
Indeed, we have that $x_2$ divides $\frac{M_2}{\text{gcd}(M_2, (x_1x_{n-1})^k)}$, but equality would imply $M_2=D$, which is impossible in our case. This, as well as degree considerations, implies that 
\begin{displaymath}
    M:=\frac{(x_1x_{n-1})^k}{\text{gcd}((x_1x_{n-1})^k,M_2)}\notin\{x_1,x_2\dots , x_n\}.
\end{displaymath}

Since $Q^L(M_2)$ is generated by a subset of the variables $\{x_1,\dots,x_n\}$ and since $M\not\in\{x_1,\dots,x_n\}$, $M$ is not part of a minimal generating set for $Q^L(M_2)$.  It follows that $Q^S(M_2)=Q^L(M_2)$, and thus $Q^S(M_2)$ is generated by a subset of the variables, as desired. \vspace{0.15cm}

\noindent\textbf{Subcase 2.2:} Next we assume that $M_2=(x_1x_{n-1})^k$.  By the proof of Subcase 1.1, it follows that $(x_2, x_3,\dots , x_{n-2})$ is contained in the ideal quotient $Q^S((x_1x_{n-1})^k)$. Pairing this with the fact that $M_1$ is in $S$, and so $x_n$ cannot divide $M_1$, we may assume without loss of generality that
\begin{align*}
    \text{supp}\left(\frac{M_1}{\text{gcd}(M_1, M_2)}\right)\subseteq\{x_1,x_{n-1}\}.
\end{align*}
However, this is a clear contradiction since $M_2=(x_1x_{n-1})^k$.\vspace{0.15cm}

\noindent\textbf{Subcase 2.3:} Finally we assume that $M_2=D$ is the distinguished generator.  First note that $(x_3,\dots,x_{n-2})$ is contained in $Q^S(M_2)$.  Indeed, for $i\in\{3,\dots,n-2\}$ each generator
\begin{align*}
    N_i:=M_2\cdot\frac{x_1x_i}{x_1x_{n-1}},
\end{align*}
precedes $M_2$ in $S$.  Thus the containment follows from the following observation for $i\in\{3,\dots , n-2\}$: 
\begin{align*}
    \frac{N_i}{\mathrm{gcd}\,(N_i,M_2)}=x_i\in Q^S(M_2).
\end{align*}

So we may assume without loss of generality that 
\begin{align*}
    \text{supp}\left(\frac{M_1}{\text{gcd}(M_1, M_2)}\right)\subseteq\{x_1,x_2,x_{n-1}\},
\end{align*} 
where, as before, we can disregard $x_n$ since $M_1\in S$. 
Thus, since $x_1x_2$ is not an edge in $\mathcal{A}_n$, we have 
\begin{align*}
    M_1=(x_1x_{n-1})^a(x_2x_{n-1})^b=x_1^ax_2^bx_{n-1}^k,
\end{align*}
for some nonnegative integers $a$ and $b$
with $a+b=k$.  In particular we have $a\leq k$.  Furthermore since $M_1$ precedes $D=x_1^{k-1}x_2x_{n-1}^k$ in $S$, it follows that $a\geq k-1$.  Thus $a\in\{k-1,k\}$ yielding 
\begin{align*}
    M_1=x_1^{k-1}x_2x_{n-1}^k\quad\text{or}\quad M_1=x_1^{k}x_{n-1}^k.
\end{align*}
However, both of these cases are impossible, as the former implies $M_1=M_2=D$, which contradicts the assumption that $M_1$ precedes $M_2$, and the latter implies $(x_1x_{n-1})^k$ precedes the distinguished generator $D$ in $S$, which contradicts Construction \ref{anticycle-order}(3).\vspace{0.15cm}

\noindent\textbf{Case 3:} In this case we assume that $M_1,M_2\in{F}$.  The vertex decompositions of $M_1$ and $M_2$ imply
\begin{align*}
    \text{supp}\left(\frac{M_1}{\text{gcd}(M_1, M_2)}\right)=\{x_i\,|\,\alpha_i>\beta_i\}.
\end{align*}
We now proceed in four subcases depending on the possibilities for the vectors $\alpha$ and $\beta$.\vspace{0.15cm}

\noindent\textbf{Subcase 3.1:} First we assume that $\alpha_n>\beta_n$.  It follows that 
\begin{align*}
    x_n\Big|\frac{M_1}{\text{gcd}(M_1, M_2)},
\end{align*}
and so it suffices to find a generator $M_3$ preceding $M_2$ in $F$ such that 
\begin{align}\label{3.1NTS}
    \frac{M_3}{\text{gcd}(M_3, M_2)}=x_n.
\end{align}
If there is some edge $x_ax_b$ of $M_2$ with $1\leq{a}<b\leq{n-1}$ and $x_ax_b\neq{x_1x_{n-1}}$, then we may proceed precisely as in Subcase 1.2 to find the desired $M_3$, and we are done.

Thus we may assume without loss of generality that all edges of $M_2$ are either equal to $x_1x_{n-1}$ or incident to $x_n$. Since $\alpha_n>\beta_n$, it follows that at least one edge of $M_2$ is not incident to $x_n$, and thus $x_1x_{n-1}$ is an edge of $M_2$.  This implies that $\{x_2,\dots,x_{n-2}\}\subseteq Q^F(M_2)$; indeed, each generator 
    \begin{displaymath}
N_i:= 
    \begin{cases}
    M_2\cdot\frac{x_ix_{n-1}}{x_1x_{n-1}}, & \text{if}\,\,\, i=2 \\
    M_2\cdot\frac{x_1x_i}{x_1x_{n-1}}, & \text{if}\,\,\, 2<i<n-1 
\end{cases}
\end{displaymath}
for $i\in\{2,\dots,n-2\}$ precedes $M_2$ in $F$ by Construction \ref{anticycle-order}(1), and thus for all $i\in\{2,\dots , n-2\}$,
\begin{align*}
    \frac{N_i}{\mathrm{gcd}\,(N_i,M_2)}=x_i\in Q^F(M_2).
\end{align*}

So we may assume without loss of generality that $\beta_i\geq{\alpha_{i}}$ for all $i\in\{2,\dots,n-2\}$.  Notice however that since $M_2$ has only the edges $x_1x_{n-1}$ and edges incident to $x_n$, it follows that $\beta_1+\beta_{n-1}+2\beta_n=2k$, which yields the following inequalities
\begin{align}\label{ineq}
    \beta_{1}+\beta_{n-1}+\beta_{n}&=2k-\beta_n
    >2k-\alpha_n=\alpha_1+\dots+\alpha_{n-1}\geq \alpha_1+\alpha_{n-1}+\alpha_n,
\end{align}
where the strict inequality follows from our assumption in this subcase and the other inequality follows directly from the structure of the anticycle graph ($x_n$ is only connected by edges to vertices $x_2,\dots,x_{n-2}$).  The strict inequality $\beta_{1}+\beta_{n-1}+\beta_{n}>\alpha_1+\alpha_{n-1}+\alpha_n$ in \eqref{ineq} together with the assumption that $\beta_i\geq\alpha_i$ for all $i\in\{2,\dots,n-2\}$ yields the inequality
\begin{align*}
    2k=\sum_{i=1}^n\beta_i>\sum_{i=1}^n\alpha_i=2k,
\end{align*}
which is a clear contradiction.  This completes this subcase.\vspace{0.15cm}

\noindent\textbf{Subcase 3.2:} Next we assume that $\alpha_n=\beta_n$ and $\alpha_2>\beta_2$.  It follows that 
\begin{align*}
    x_2\Big|\frac{M_1}{\text{gcd}(M_1, M_2)},
\end{align*}
and so it suffices to find a generator $M_3$ preceding $M_2$ in $F$ such that 
\begin{align}\label{3.2NTS}
    \frac{M_3}{\text{gcd}(M_3, M_2)}=x_2.
\end{align}
If there is some edge $x_ax_b$ of $M_2$ such that $3\leq{a}<b\leq{n-1}$, then the generator
\begin{align*}
    M_3=M_2\cdot\frac{x_2x_b}{x_ax_b}
\end{align*}
precedes $M_2$ in $F$ and satisfies \eqref{3.2NTS}, as desired.  

Thus we may assume that all edges of $M_2$ are incident to $x_1$, $x_2$, or $x_n$. If $x_1x_b$ is an edge of $M_2$ for some $4\leq{b}\leq{n-1}$, we find that the generator
\begin{align*}
    M_3=M_2\cdot\frac{x_2x_b}{x_1x_b}
\end{align*}
precedes $M_2$ in $F$ and satisfies \eqref{3.2NTS}, as desired.  Similarly, if $x_ax_n$ is an edge of $M_2$ for some $3\leq{a}\leq{n-2}$, then the argument is completed by choosing
\begin{align*}
    M_3=M_2\cdot\frac{x_2x_n}{x_ax_n}.
\end{align*}

Thus we may assume that every edge of $M_2$ is either incident to $x_2$ or equal to $x_1x_3$.  Since $\alpha_2>\beta_2$, it follows that at least one edge of $M_2$ is not incident to $x_2$, and thus $x_1x_{3}$ is an edge of $M_2$.  Furthermore, since $M_2\in{F}$, it follows that $x_2x_n$ is also an edge of $M_2$.  Thus the generator
\begin{align*}
    M:=M_2\cdot\frac{(x_2x_4)(x_3x_n)}{(x_1x_3)(x_2x_n)}
\end{align*}
precedes $M_2$ in $F$, yielding
\begin{align*}
    \frac{M}{\mathrm{gcd}\,(M,M_2)}=x_4\in Q^F(M_2).
\end{align*}
Furthermore, we have that $\{x_5,\dots,x_{n}\}\subseteq Q^F(M_2)$. Indeed, each generator
\begin{align*}
    N_i=M_2\cdot\frac{x_3x_i}{x_1x_3}
\end{align*}
for $i\in\{5,\dots , n\}$ precedes $M_2$ in $F$ by Construction \ref{anticycle-order}(1),
and thus
\begin{align*}
    \frac{N_i}{\mathrm{gcd}\,(N_i,M_2)}=x_i\in Q^F(M_2).
\end{align*}

Therefore we may further assume that $\beta_i\geq{\alpha_i}$, for all $i\in\{4,\dots , n\}$.  However, by the same argument as in Subcase 3.1, the equality $\beta_1+\beta_3+2\beta_2=2k$ implies that $\beta_1+\beta_2+\beta_3>\alpha_1+\alpha_2+\alpha_3$, which coupled with the fact that $\beta_i\geq{\alpha_i}$ for all $i\in\{4,\dots , n\}$, produces the same contradiction as in Subcase 3.1 and completes this subcase. \vspace{0.15cm}

\noindent\textbf{Subcase 3.3:} Next we assume that $\alpha_n=\beta_n$, $\alpha_i=\beta_i$, for all $i\in\{2,\dots , j-1\}$, and $\alpha_j>\beta_j$, for some $j\in\{3,\dots , n-2\}$.  It follows that 
\begin{align*}
    x_j\Big|\frac{M_1}{\text{gcd}(M_1, M_2)},
\end{align*}
so it suffices to find a generator $M_3$ preceding $M_2$ in $F$ such that 
\begin{align}\label{3.3NTS}
    \frac{M_3}{\text{gcd}(M_3, M_2)}=x_j.
\end{align}
If there is an edge $x_ax_b$ of $M_2$ such that $1\leq{a}\leq{j-2}$ and $j+1\leq{b}\leq{n-1}$, then the generator
\begin{align*}
    M_3=M_2\cdot\frac{x_ax_j}{x_ax_b}
\end{align*}
precedes $M_2$ in $F$ and satisfies \eqref{3.3NTS}, as desired.  Similarly, if there is an edge $x_ax_b$ of $M_2$ such that $j+1\leq{a}<b\leq{n}$, then the argument is completed by choosing
\begin{align*}
    M_3=M_2\cdot\frac{x_jx_b}{x_ax_b}.
\end{align*}
And if $x_1x_c$ is an edge of $M_2$ for some $3\leq{c}\leq{j-2}$, then the argument is completed by choosing
\begin{align*}
      M_3=M_2\cdot\frac{x_cx_j}{x_1x_c}.
\end{align*}

So in summary, we may assume that all edges of $M_2$ are incident to $x_{j-1}$ or $x_{j}$, are of the form $x_cx_{n}$, for some $2\leq{c}\leq{j-2}$, or are of the form $x_ax_b$ with $2\leq{a}<b\leq{j-2}$. By our hypothesis for this subcase, there is an integer $r\in\{1\}\cup\{j+1,\dots , n-1\}$ such that $\alpha_r<\beta_r$, and in particular, $\beta_r>0$. Thus by the preceding description of edges of $M_2$, it follows that $x_rx_{z}$ is an edge of $M_2$ for some $z\in\{j-1, j\}$, and furthermore that every edge of $M_2$ incident to $x_r$ must also be incident to $x_{j-1}$ or $x_{j}$.  This allows us to rule out edges of the form $x_ax_b$ with $2\leq{a}<b\leq{j-2}$ in $M_2$.  Indeed, suppose $x_ax_b$ is such an edge. Since $2\leq{a}<j-2$ and $2<b\leq{j-2}$, we have that $x_ax_z,x_bx_r\in{I(\mathcal{A}_n)}$. Thus $M_2$ has the following edge decomposition: 
\begin{align*}
    M_2=M_2\cdot\frac{(x_ax_z)(x_bx_r)}{(x_ax_b)(x_rx_z)},
\end{align*}
which implies that $x_bx_r$ is an edge of $M_2$, but this contradicts the fact that all edges of $M_2$ incident to $x_r$ must also be incident to $x_{j-1}$ or $x_{j}$. 

Now we may assume that all edges of $M_2$ are incident to $x_{j-1}$ or $x_{j}$, or are of the form $x_cx_{n}$ for some $2\leq{c}\leq{j-2}$.  In fact, there must be an edge of the third type.  Otherwise, if all edges of $M_2$ are incident to $x_j$ or $x_{j-1}$, then our hypotheses in this subcase imply
\begin{align*}
    k=\beta_{j-1}+\beta_j=\alpha_{j-1}+\beta_j<\alpha_{j-1}+\alpha_j,
\end{align*}
and it follows that the degree of $M_1$ is at least $2(\alpha_{j-1}+\alpha_j)>2k$, which is a contradiction.  Thus there must be an edge of $M_2$ of the form $x_cx_n$ for some $2\leq c\leq j-2$.

Next recall that $x_rx_z$ is an edge of $M_2$ for some $r\in\{1\}\cup\{j+1,\dots , n-1\}$ and $z\in\{j-1, j\}$. In fact, we must have $r=1$ and $c=2$. Indeed, if $r\neq{1}$ or $c\neq{2}$, then $M_2$ has the edge decomposition 
\begin{align*}
    M_2=M_2\cdot\frac{(x_cx_r)(x_zx_n)}{(x_cx_n)(x_rx_z)},
\end{align*}
which implies that $x_cx_r$ is an edge of $M_2$, but this contradicts our description of the edges of $M_2$ given that $c,r\notin\{j-1, j,n\}$.  Thus $r=1$ and $c=2$.

Thus we may assume that all edges of $M_2$ are incident to $x_{j-1}$ or $x_{j}$, or are of the form $x_2x_n$, and that $x_2x_n$ and $x_1x_z$ with $z\in\{j-1,j\}$ are edges of $M_2$.  Now it follows that $j>3$.  Indeed, if $j=3$, then $j-1=2$,  and it follows that all edges of $M_2$ are incident to $x_{j-1}$ or $x_{j}$, which we have already shown is impossible.

So we may choose the generator 
\begin{align*}
    M_3=M_2\cdot\frac{(x_2x_j)(x_zx_n)}{(x_1x_z)(x_2x_n)},
\end{align*}
which precedes $M_2$ in $F$ and satisfies \eqref{3.3NTS}, as desired.\vspace{0.15cm}

\noindent\textbf{Subcase 3.4:}  Finally we assume that $\alpha_n=\beta_n$, $\alpha_i=\beta_i$ for all $i\in\{2,\dots, n-2\}$, and $\alpha_{n-1}>\beta_{n-1}$.  Note that this is indeed the final case because if $\alpha_n=\beta_n$ and $\alpha_i=\beta_i$ for all $i\in\{2,\dots , n-1\}$, then $M_1=M_2$, which contradicts  our assumption that $M_1$ precedes $M_2$ in $F$.  

In this case, notice that since $M_1$ and $M_2$ both have degree $2k$, our assumptions in this subcase imply that $\beta_1>\alpha_1$.
In particular we have $\beta_1>0$.  
Since $\alpha_{n-1}>\beta_{n-1}$, it also follows that 
\begin{align*}
    x_{n-1}\Big|\frac{M_1}{\text{gcd}(M_1, M_2)},
\end{align*}
and so it suffices to find a generator $M_3$ preceding $M_2$ in $F$ such that 
\begin{align}\label{3.4NTS}
    \frac{M_3}{\text{gcd}(M_3, M_2)}=x_{n-1}.
\end{align}

Suppose for the sake of contradiction that there is no such generator $M_3$.  Since $\beta_1>0$, we have that $x_1x_z$ is an edge of $M_2$ for some $z\in\{3,\dots,n-1\}$.  In fact we must have $z\in\{n-2,n-1\}$, because otherwise the generator
\begin{align*}
    M_3=M_2\cdot\frac{x_zx_{n-1}}{x_1x_z}
\end{align*}
precedes $M_2$ in $F$ and satisfies \eqref{3.4NTS}, yielding a contradiction.  In particular, every edge incident to $x_1$ must also be incident to $x_{n-2}$ or $x_{n-1}$.

Now we claim that there are no edges $x_ax_b$ of $M_2$ with  $2\leq{a}<b\leq{n-3}$.  This is indeed the case because otherwise, $M_2$ has the following edge decomposition
\begin{align*}
    M_2=M_2\cdot\frac{(x_1x_b)(x_ax_z)}{(x_ax_b)(x_1x_z)},
\end{align*} 
and thus $x_1x_b$ is an edge of $M_2$ with  $b\notin\{n-2,n-1\}$, which is impossible by the argument above.  Thus our claim holds, and since every edge incident to $x_1$ must also be incident to $x_{n-2}$ or $x_{n-1}$, it follows that every edge of $M_2$ is incident to $x_{n-2}, x_{n-1}$, or $x_n$.

Next we rule out $x_1x_{n-2}$ as an edge of $M_2$.  Supposing to the contrary that $x_1x_{n-2}$ is an edge of $M_2$, we claim that in this case $x_ax_n$ is not an edge of $M_2$ for $2\leq{a}\leq{n-3}$.  Indeed, if $x_ax_n$ is an edge of $M_2$ for some $3\leq{a}\leq{n-3}$, then $M_2$ has the following edge decomposition
\begin{align*}
    M_2=M_2\cdot\frac{(x_1x_a)(x_{n-2}x_n)}{(x_1x_{n-2})(x_ax_n)},
\end{align*}
and thus $x_1x_a$ is an edge of $M_2$ with $a\notin\{n-2,n-1\}$, which is impossible.  On the other hand, if $x_2x_n$ is an edge of $M_2$, then the generator 
\begin{align*}
    M_3=M_2\cdot\frac{(x_2x_{n-1})(x_{n-2}x_n)}{(x_1x_{n-2})(x_2x_n)}
\end{align*}
precedes $M_2$ in $F$ and  satisfies \eqref{3.4NTS}, yielding another contradiction.  This proves the claim, and it follows that any edge of $M_2$ incident to $x_n$ must also be incident to $x_{n-2}$.  Thus every edge of $M_2$ must be incident to $x_{n-2}$ or $x_{n-1}$.  However, since $\alpha_{n-2}=\beta_{n-2}$ and $\alpha_{n-1}>\beta_{n-1}$, this implies that 
\begin{align*}
    k=\beta_{n-1}+\beta_{n-2}<\alpha_{n-1}+\alpha_{n-2},
\end{align*}
which is impossible given the structure of $\mathcal{A}_n$.  Therefore, $x_1x_{n-2}$ is not an edge of $M_2$, as desired, and in particular $x_1x_{n-1}$ must be an edge of $M_2$.

Finally suppose that $M_2$ has edges $x_{a}x_{n-2}$ and $x_{b}x_{n}$, for some $a,b\in\{1,2,\dots, n-3\}$.  Then we have that the generator  
\begin{align*}
    M_3=M_2\cdot\frac{(x_{n-2}x_n)(x_{a}x_{n-1})(x_{b}x_{n-1})}{(x_{a}x_{n-2})(x_{b}x_{n})(x_1x_{n-1})}
\end{align*}
precedes $M_2$ in $F$ and satisfies \eqref{3.4NTS}, which yields a contradiction.  It follows that either all edges of $M_2$ incident to $x_{n-2}$ must also be incident to $x_{n}$, or all edges of $M_2$ incident to $x_{n}$ must also be incident to $x_{n-2}$.  Using the fact that every edge of $M_2$ must be incident to $x_{n-2}, x_{n-1}$, or $x_n$, the former case implies that all edges of $M_2$ are incident to $x_{n-1}$ or $x_n$, and the latter case implies that all edges of $M_2$ are incident to $x_{n-1}$ or $x_{n-2}$.  In either case, we have a contradiction to our hypotheses that $\alpha_{n-1}>\beta_{n-1}$, $\alpha_n=\beta_n$, and $\alpha_{n-2}=\beta_{n-2}$.  This completes the proof.
\end{proof}

\begin{remk}
    Another strategy for proving Theorem \ref{anticycle-theorem} is suggested in the subsequent preprint \cite{EFHHKM} (see \cite[Remark 6.8]{EFHHKM}).
\end{remk}

\section{Linear quotient computations}

In this section we examine the linear quotient property from a computational perspective using new methods we constructed on Macaulay2.  In particular, we introduce three methods: \texttt{isLinear}, \texttt{getQuotients}, and \texttt{findLinearOrderings}.  To describe these methods and provide examples, we adopt Notation \ref{quotientnotation} for ideal quotients to be used throughout this section.

Our first two methods take as an input an ordered list $\mathcal{O}=(m_1,\ldots,m_r)$ whose entries minimally generate an ideal $I$ of a standard graded polynomial ring.  The method \texttt{isLinear} tests whether $\mathcal{O}$ is a linear quotient ordering for $I$, and the method \texttt{getQuotients} outputs the ordered list of ideal quotients $\left(Q^\mathcal{O}(m_1), Q^\mathcal{O}(m_2),\ldots, Q^\mathcal{O}(m_r)\right)$ corresponding to the ordering $\mathcal{O}$.  The method \texttt{findLinearOrderings} takes an ideal $I$ as an input and returns an ordered list $\mathcal{O}$ of the generators of $I$ that is what we call \textit{most linear}; see Definition \ref{defmostlinear}.  The code for each of these methods and further documentation is in \cite{MSM24}.  

Our methods complement those in \cite{ficarra}; the method \texttt{isLinear} offers similar capabilities, although \texttt{getQuotients} and \texttt{findLinearOrderings} offer new capabilities beyond those in \cite{ficarra}.  In particular, \texttt{findLinearOrderings} provides useful homological information about edge ideals, even those that do not admit linear quotients; see Definition \ref{defmostlinear}, Example \ref{findorderex}, and Example \ref{findorderex2} for a precise description of its capabilities.   

To illustrate our results in Section 5, we focus on powers of the anticycle on six vertices $\mathcal{A}_6$, pictured below, first using our methods \verb|isLinear| and \verb|getQuotients|. We input the list of generators ordered according to the linear quotient ordering in Construction \ref{anticycle-order} (see also Theorem \ref{anticycle-theorem}) and verify that indeed the given ordering yields linear quotients.

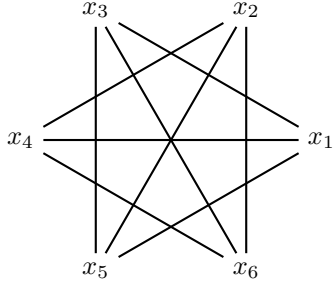
\begin{figure}[hbt]
\begin{center}
\begin{tikzpicture}[scale = 1]
\node (a) at (2, 0) {$x_1$};
\node (b) at (1, {sqrt(3)}) {$x_2$};
\node (c) at (-1, {sqrt(3)}) {$x_3$};
\node (d) at (-2, 0) {$x_4$};
\node (e) at (-1, {-sqrt(3)}) {$x_5$};
\node (f) at (1, {-sqrt(3)}) {$x_6$};
\path [thick]
(a) edge node {} (c)
(a) edge node {} (d)
(a) edge node {} (e)
(b) edge node {} (d)
(b) edge node {} (e)
(b) edge node {} (f)
(c) edge node {} (e)
(c) edge node {} (f)
(d) edge node {} (f);
\end{tikzpicture}
\end{center}
\label{fig:anticycle}

\end{figure}

\begin{ex}\label{getquotientex}(\verb|getQuotients| and \verb|isLinear| with $I(\mathcal{A}_6)^2$)\\

We start by considering the ordering $\mathcal{O}_6^{(2)}$ defined in Construction \ref{anticycle-order}; this is a linear quotient ordering by Theorem \ref{anticycle-theorem}.  We first define our polynomial ring in six variables, and then create the list $\mathcal{O}_6^{(2)}$ by defining two sublists, $F$ and $S$, where $F$ consists of all generators divisible by $x_6$, and $S$ consists of all generators not divisible by $x_6$, as is described in Construction \ref{anticycle-order}.

\begin{verbatim}

i1 : Q = QQ[x_1..x_6];

i2 : F = {x_6^2*x_2^2, x_6^2*x_2*x_3, x_6^2*x_2*x_4, x_6^2*x_3^2, x_6^2*x_3*x_4,
          x_6^2*x_4^2, x_6*x_2^2*x_4, x_6*x_2^2*x_5, x_6*x_2*x_3*x_4,
          x_6*x_2*x_3*x_5, x_6*x_2*x_3*x_1, x_6*x_2*x_4^2, x_6*x_2*x_4*x_5,
          x_6*x_2*x_4*x_1, x_6*x_2*x_5*x_1, x_6*x_3^2*x_5, x_6*x_3^2*x_1,
          x_6*x_3*x_4*x_5, x_6*x_3*x_4*x_1, x_6*x_3*x_5*x_1, x_6*x_4^2*x_1,
          x_6*x_4*x_5*x_1};

i3 : S = {x_1^2*x_3^2, x_1^2*x_3*x_4, x_1^2*x_3*x_5, x_1^2*x_4^2, x_1^2*x_4*x_5,
          x_1*x_2*x_3*x_4, x_1*x_2*x_3*x_5, x_1*x_2*x_4^2, x_1*x_2*x_4*x_5,
          x_1*x_2*x_5^2, x_1^2*x_5^2, x_1*x_3^2*x_5, x_1*x_3*x_4*x_5,
          x_1*x_3*x_5^2, x_2^2*x_4^2, x_2^2*x_4*x_5, x_2^2*x_5^2,
          x_2*x_3*x_4*x_5, x_2*x_3*x_5^2, x_3^2*x_5^2};
          
\end{verbatim}
Now we concatenate these two lists to get $\mathcal{O}_6^{(2)}$, and pass it as input into \verb|getQuotients|:
\begin{verbatim}
i4 : getQuotients (F | S)

o4 = {{x }, {x , x }, {x }, {x , x }, {x , x }, {x }, {x , x }, {x , x },
        2     3   2     2     3   2     3   2     6     6   4     6   2  
     ---------------------------------------------------------------------------
      {x , x , x }, {x , x , x }, {x , x , x }, {x , x , x , x },
        6   4   2     6   5   4     6   3   2     6   4   3   2
     ---------------------------------------------------------------------------
      {x , x , x , x , x }, {x , x , x }, {x , x }, {x , x , x }, {x , x , x },
        6   5   4   3   2     4   3   2     6   2     6   5   2     6   3   2
     ---------------------------------------------------------------------------
      {x , x , x , x }, {x , x , x }, {x , x , x }, {x , x , x }, {x }, {x , x },
        6   5   3   2     4   3   2     6   3   2     4   3   2     6     6   3
     ---------------------------------------------------------------------------
      {x , x , x }, {x , x }, {x , x , x }, {x , x }, {x , x , x }, {x , x , x },
        6   4   3     6   3     6   4   3     6   1     6   4   1     6   3   1
     ---------------------------------------------------------------------------
      {x , x , x , x }, {x , x , x }, {x , x , x }, {x , x , x },
        6   4   3   1     6   4   3     4   3   2     6   2   1
     ---------------------------------------------------------------------------
      {x , x , x , x }, {x , x , x , x , x }, {x , x }, {x , x , x },
        6   3   2   1     6   4   3   2   1     6   1     6   4   1
     ---------------------------------------------------------------------------
      {x , x , x }, {x , x , x }, {x , x , x , x }, {x , x , x }}
        6   4   1     6   2   1     6   4   2   1     6   2   1

o4 : List

\end{verbatim}

Inspecting the output, we notice that each quotient is generated by vertices in $\mathcal{A}_6$, and thus $\mathcal{O}_6^{(2)}$ is a linear quotient ordering as we proved in Theorem \ref{anticycle-theorem}. To verify this result computationally, we use the method \verb|isLinear|:
\begin{verbatim}

i5 : isLinear (F | S, n = 6)

o5 = true

\end{verbatim}
\end{ex}

Next we introduce the notion of a \textit{most linear ordering} for an ideal $I$.

\begin{defn}\label{defmostlinear}
    Given an ordered list $\mathcal{O}=(m_1,\ldots,m_r)$ whose entries minimally generate an ideal $I$ of a standard graded polynomial ring, we say that $\mathcal{O}$ is \textit{linear up to $n$} if its ideal quotients $Q^\mathcal{O}(m_i)$ are linear for all $i\leq n$. If $\mathcal{O}$ is linear up to $n$, and no permutation of $\mathcal{O}$ is linear up to $j>n$, then we call $\mathcal{O}$ a \textit{most linear ordering} for the ideal $I$.
\end{defn}

To find a linear quotient ordering of an edge ideal $I(G)$, or a \textit{most linear} ordering in the case that $I(G)$ does not admit linear quotients, we use the method \verb|findLinearOrderings|.

\begin{ex}(\verb|findLinearOrderings| with $I(\mathcal{A}_6)$)\label{findorderex}
\begin{verbatim}
i6 : antiCycleSix = ideal (x_1*x_3, x_1*x_4, x_1*x_5, x_2*x_4, x_2*x_5,
                           x_2*x_6, x_3*x_5, x_3*x_6, x_4*x_6);

i7 : findLinearOrderings (antiCycleSix, 6)

No linear ordering found; returning most linear ordering as a list.
    -- Elapsed time: .16227 seconds.

o7 = {x x , x x , x x , x x , x x , x x }
        4 6   3 6   2 6   3 5   2 5   2 4

o7 : List

\end{verbatim}

As expected from \cite{Fro90} (see also \cite{HW11}), we observe that $I(\mathcal{A}_6)$ does not admit linear quotients.  Furthermore, we find a most linear ordering for $\mathcal{A}_6$ and consequently, a subgraph $G$ of $\mathcal{A}_6$ of largest size whose edge ideal $I(G)=(x_4x_6,\, x_3x_6,\, x_2x_6,\, x_3x_5,\, x_2x_5,\,x_2x_4)$ admits linear quotients.  This gives some indication of how close $I(\mathcal{A}_6)$ is to admitting linear quotients.

In the next example we consider the second power of $I(\mathcal{A}_6)$, which admits linear quotients as proved in Theorem \ref{anticycle-theorem}.  We use \texttt{findLinearOrderings} to exhibit a linear quotient ordering.
\end{ex}

\begin{ex}(\verb|findLinearOrderings| with $I(\mathcal{A}_6)^2$)\label{findorderex2}
\begin{verbatim}
i8 : findLinearOrderings (trim antiCycleSix^2, 6)

Linear ordering found, returning as a list.
    -- Elapsed time: 1.54978 seconds.

       2 2       2       2   2 2       2   2 2                       2    
o8 = {x x , x x x , x x x , x x , x x x , x x , x x x x , x x x x , x x x ,
       4 6   3 4 6   2 4 6   3 6   2 3 6   2 6   3 4 5 6   2 4 5 6   3 5 6
     ---------------------------------------------------------------------------
                           2                   2
      x x x x , x x x x , x x x , x x x x , x x x , x x x x , x x x x ,
       2 3 5 6   1 3 5 6   2 5 6   1 2 5 6   2 4 6   2 3 4 6   1 3 4 6
     ---------------------------------------------------------------------------
         2               2                   2               2 2       2
      x x x , x x x x , x x x , x x x x , x x x , x x x x , x x , x x x ,
       1 4 6   1 4 5 6   2 4 6   1 2 4 6   1 3 6   1 2 3 6   3 5   2 3 5
     ---------------------------------------------------------------------------
          2    2 2       2                       2                 2       2 2
      x x x , x x , x x x , x x x x , x x x x , x x x , x x x x , x x x , x x ,
       1 3 5   2 5   1 2 5   2 3 4 5   1 3 4 5   2 4 5   1 2 4 5   1 4 5   1 5
     ---------------------------------------------------------------------------
         2               2       2 2       2   2 2             2       2 2
      x x x , x x x x , x x x , x x , x x x , x x , x x x x , x x x , x x }
       1 3 5   1 2 3 5   1 3 5   2 4   1 2 4   1 4   1 2 3 4   1 3 4   1 3

o8 : List

\end{verbatim}
\end{ex}
As desired we find a linear quotient ordering of $I(A_6)^2$. Moreover, this ordering is different from $O_6^{(2)}$ (see Example \ref{getquotientex}), demonstrating that linear quotient orderings need not be unique.

\section*{Acknowledgments}
The authors thank Eran Nevo for helpful discussions and his suggestion to work on Theorem \ref{anticycle-theorem}.  They also thank Takayuki Hibi and Alessio D'Ali for shedding light on the history and origin of Corollary \ref{recover}.
Finally they thank Allen Mi, Max Hammond, and Hamilton Wan for fruitful discussions and Sinisa Pajevic for editorial assistance at the initial stages of this project. 




\bibliographystyle{siam}
\bibliography{linearquotients}
 \end{document}